\documentclass[aap,MSNbibl,seceqn,dvips]{arximspdf}
\usepackage{graphicx}
\usepackage{subenv}

\doi{10.1214/12-AAP895} 
\volume{23}
\issue{5}
\pubyear{2013}
\firstpage{2099}
\lastpage{2138}

\makeatletter
\newcommand{\rrvert}{\vert}
\newcommand{\llvert}{\vert}

\newcommand{\implies}{\Longrightarrow}

\newtheorem{theorem}{Theorem}[section]
\newtheorem{lemma}[theorem]{Lemma}
\newtheorem{corollary}[theorem]{Corollary}

\newproclaim{remark}[theorem]{Remark}

\newcommand{\BR}{{\mathbb R}}

\newcommand{\BE}{{\mathbb E}}
\newcommand{\BP}{{\mathbb P}}

\newcommand{\CM}{{\mathcal M}}
\newcommand{\CL}{{\mathcal L}}
\newcommand{\CS}{{\mathcal S}}
\newcommand{\CC}{{\mathcal C}}
\newcommand{\CJ}{{\mathcal J}}
\newcommand{\CI}{{\mathcal I}}
\newcommand{\CE}{{\mathcal E}}

\newcommand{\one}{{\mathbf1}}

\newcommand{\trace}{\operatorname{trace}}

\newcommand{\weakto}{\stackrel{w}{\rightarrow}}

\makeatother

\begin{document}
\begin{frontmatter}

\title{Systems with large flexible server pools: Instability of ``natural'' load balancing}
\runtitle{Load balancing instability}

\begin{aug}
\author[A]{\fnms{Alexander L.} \snm{Stolyar}\corref{}\ead[label=e1]{stolyar@research.bell-labs.com}}
\and
\author[B]{\fnms{Elena} \snm{Yudovina}\thanksref{t1}\ead[label=e2]{yudovina@umich.edu}}
\runauthor{A. L. Stolyar and E. Yudovina}
\affiliation{Bell Labs, Alcatel-Lucent and University of Michigan}
\address[A]{Bell Labs, Alcatel-Lucent\\
600 Mountain Ave., 2C-322\\
Murray Hill, New Jersey 07974\\
USA\\
\printead{e1}}
\address[B]{Department of Statistics\\
University of Michigan\\
439 West Hall, 1085 South University Ave.\\
Ann Arbor, Michigan 48109\\
USA\\
\printead{e2}} 
\end{aug}

\thankstext{t1}{Supported by the NSF Graduate Research Fellowship.}

\received{\smonth{12} \syear{2010}}
\revised{\smonth{4} \syear{2012}}

%
\begin{abstract}
We consider general large-scale service systems with multiple customer
classes and multiple server (agent) pools,
mean service times depend both on the customer class and server pool.
It is assumed that the allowed activities
(routing choices) form a tree (in the graph with vertices
being both customer classes and server pools).
We study the behavior of the system under a natural (load balancing)
routing/scheduling rule, \textit{Longest-Queue Freest-Server}
(LQFS-LB), in
the many-server asymptotic regime, such that the
exogenous arrival rates of the customer classes, as well as the
number of agents in each pool, grow to infinity in proportion to
some scaling parameter $r$. \textit{Equilibrium point} of the system under
LQBS-LB is the desired operating point, with server pool loads
minimized and perfectly balanced.

Our main results are as follows. (a) We show that, quite surprisingly
(given the tree assumption), for certain parameter ranges, the
\textit{fluid limit} of the system may be \textit{unstable} in the
vicinity of the
equilibrium point; such instability may occur if the activity graph is
not ``too small.'' (b) Using (a), we demonstrate that the sequence of
stationary distributions of \textit{diffusion-scaled} processes [measuring
$O(\sqrt{r})$ deviations from the equilibrium point] may be nontight,
and in fact may escape to infinity. (c) In one special case of
interest, however, we show that the sequence of stationary
distributions of diffusion-scaled processes is tight, and the limit of
stationary distributions is the stationary distribution of the limiting
diffusion process.
\end{abstract}

%

\end{frontmatter}

\section{Introduction}

Large-scale service systems (such as call centers)
with heterogeneous customer and server (agent) populations bring up the
need for efficient dynamic control policies that
match arriving (or waiting) customers and available servers. In this
setting, two goals are desirable. On the one hand, customers should not
be kept waiting, if this is possible. On the other hand, idle time
should be distributed fairly among the servers. For example, one would
like to avoid the situation in which one of the server pools is fully
busy while another one has significant numbers of idle agents.

Consider a general system, where the arrival rate of class $i$
customers is $\Lambda_i$,
the service rate of a class $i$ customer by type $j$ agent is $\mu_{ij}$,
and the server pool sizes are $B_j$.
Another very desirable feature of a dynamic control is insensitivity to
parameters $\Lambda_i$ and $\mu_{ij}$.
That is, the assignment of customers to server pools should, to the
maximal degree possible, depend only on the current system state, and
not on prior knowledge of arrival rates or mean service times, because
those parameters may not be known in advance and, moreover, they may be
changing in time.

If the system objective is to minimize the maximum average load of any
server pool, a ``static'' optimal control can be obtained by solving a
linear program, called \textit{static planning problem} (SPP), which
has $B_j$'s, $\mu_{ij}$'s and $\Lambda_i$'s as parameters. An optimal
solution to the SPP will prescribe optimal average rates $\Lambda_{ij}$
at which arriving customers should be routed to the server pools.
Typically (in a certain sense) the solution to SPP is unique and the
\textit{basic activities}, that is, routing choices $(ij)$ for which
$\Lambda_{ij}>0$, form a tree; let us assume this is the case. It is
possible to design a dynamic control policy, which achieves the load
balancing objective without a priori knowledge of input rates
$\Lambda_i$---the \textit{Shadow Routing} policy in~\cite{StolyarTezcanunderload,StolyarTezcan}
does just that, and in the process it ``automatically identifies'' the
basic activity tree. Shadow Routing policy, however, does need to
``know'' the service rates $\mu_{ij}$.

The key question we address in this paper is as follows. Suppose a
control policy
does \textit{not} know the service rates $\mu_{ij}$, but ``somehow'' it
does know
the structure of the basic activity tree, and restrict
routing to this tree only. [E.g., all feasible activities,
i.e., those $(ij)$'s for which $\mu_{ij}>0$, may form a tree simply
by the structure of the system.
Another example: if Shadow Routing has some estimates of $\mu_{ij}$,
this will not
be sufficient
for it to identify the optimal routing rates, but may very well be
sufficient to
correctly identify the
basic activity tree.] What is an efficient load balancing policy in
this case?

If routing is restricted to a tree, it is very natural to conjecture that
simple policies of the type considered by Gurvich and Whitt \cite
{GurvichWhitt},
Atar, Shaki and Shwartz~\cite{Atar2009} and Armony and Ward \cite
{ArmonyWard},
which are of the ``serve longest queue'' and ``join least loaded pool'' type,
should ``typically be good enough.'' Some of the results in these (and
other) papers,
in fact, prove optimal behavior of simple load balancing
schemes on a \textit{finite time interval};
which further supports the above informal conjecture.
One of the main contributions of our work is to show that,
surprisingly, the above conjecture
is \textit{not} correct for a general parameter setting. The key reason
is that
a ``natural'' load balancing, \textit{even if it is done along an
a priori given optimal tree},
may render the system unstable in the vicinity
of equilibrium point.

The specific
control rule we analyze in this paper can be seen as a special case of
the Queue-and-Idleness Ratio rule considered in~\cite{GurvichWhitt}.
Within the given (basic) activity tree, if an arriving customer sees
multiple available servers, it will choose the server pool with the
smallest load; while if a server sees several customers waiting in
queues, it will take a customer from the longest queue.
We call this rule \textit{Longest-Queue Freest-Server} (LQFS-LB).

We consider a many-server asymptotic regime, such that $\Lambda_i =
\lambda_i r$
[or sometimes $\Lambda_i = \lambda_i r + O(\sqrt{r})$], $B_j = \beta
_j r$,
where $\lambda_i$ and $\beta_j$ are some positive constants, $r\to
\infty$ is a scaling
parameter and $\mu_{ij}$ remain constant. Our key results show that
the \textit{fluid limit}
of the system process (obtained via space-scaling by $1/r$) can be
unstable in the vicinity
of the equilibrium point. This is very counterintuitive, because it
would be reasonable
to expect the contrary: that a simple load balancing in a system with
activity graph
free of cycles would be ``well behaved.''

Using the fluid limit local instability (when such occurs), we prove
that
the sequence of stationary distributions
of \textit{diffusion-scaled} processes [measuring $O(\sqrt{r})$
deviations from the equilibrium
point] may be nontight, and in fact may escape to infinity.
This of course means, in particular, that the behavior of the diffusion
limit in the vicinity
of equilibrium point
\textit{on a finite time interval}, may not be relevant to the system
behavior in steady state,
because the system ``does not spend any time'' in the $O(\sqrt{r})$-vicinity
of the equilibrium point.

In addition to the instability examples, we prove that in several cases the
fluid limit will be (at least locally) stable.
We demonstrate that fluid limit of
any underloaded system with at most two customer classes,
or critically loaded system with at most four customer classes,
is always locally stable.
We also demonstrate local stability in the case when the service rate
depends only on
the customer type (but not server pool, as long as it can serve it).
In the case when the service rate depends only on the server type (but
not customer
type, as long as it can be served), we show more---the global
stability of the fluid limit.

General results on the asymptotics of
stationary distributions (most impor\-tantly---their tightness),
especially in the many-server systems' diffusion limit,
are notoriously difficult to derive; for recent results in this
direction see~\cite{GamarnikZeevi,GamarnikMomcilovic}.
In the special case when the service rate depends only on the server type,
we prove that under the LQFS-LB policy the sequence of stationary distributions
of diffusion-scaled processes is tight,
and the limit of stationary distributions is the
stationary distribution
of the limiting diffusion process.

The structure of the paper is as follows. In Section \ref
{sectionmodel} we present the model,
define the static planning problem and related notions
and define the LQFS-LB policy.
In Section~\ref{sectionfluidmodel} we define fluid models of the system,
derive their basic properties in the vicinity of an equilibrium point,
and define local stability.
Section~\ref{sectionspecialcases} contains fluid model
stability results in the two special cases when the service rate
depends on server class only
or on customer type only.
Our key results on local instability of fluid models
are presented in Section~\ref{secunderloaddivergence}.
In Section~\ref{sectiondiffusionlimit} we consider
an underloaded system (with optimal average utilization being
$1-\varepsilon<1$), and
prove possible evanescence of stationary distributions of the diffusion scaled
processes.
Finally, Section~\ref{sec-hw} considers the so-called Halfin--Whitt asymptotic
regime [where the optimal average utilization is $1-O(1/\sqrt{r})$],
and contains two
main results on the asymptotics
of stationary distributions of the diffusion scaled
processes: (a) possible evanescence under certain parameters and (b)
tightness (and ``limit interchange'') result for the case
when the service rate depends only on the server type.

\section{Model}\label{sectionmodel}
\subsection{The model; static planning (LP) problem}
Consider the model in which there are $I$ customer classes, or types,
labeled $1,2,\ldots,I$, and $J$ server (agent) pools, or classes,
labeled $1,2,\ldots,J$ (generally, we will use the subscripts $i$,
$i'$ for customer classes, and $j$, $j'$ for server pools). The sets of
customer classes and server classes will be denoted by $\CI$ and $\CJ
$, respectively.

We are interested in the scaling properties of the system as it grows
large. The meaning of ``grows large'' is as follows. We consider a
sequence of systems indexed by a scaling parameter $r$. As $r$ grows,
the arrival rates and the sizes of the service pools, but not the speed
of service, increase. Specifically, in the $r$th system, customers of
type $i$ enter the system as a Poisson process of rate $\lambda^r_i =
r \lambda_i + o(r)$, while the $j$th server pool has $r \beta_j$
individual servers. (All $\lambda_i$ and $\beta_j$ are positive parameters.)
Customers may be accepted for service immediately upon arrival, or
enter a queue; there is a separate queue for each customer type.
Customers do not abandon the system. When a customer of type $i$ is
accepted for service by a server in pool $j$, the service time is
exponential of rate $\mu_{ij}$; the service rate depends both on the
customer type and the server type, but \textit{not} on the scaling
parameter $r$. If customers of type $i$ cannot be served by servers of
class $j$, the service rate is $\mu_{ij} = 0$.

We would like to balance the proportion of busy servers across the
server pools, while keeping the system operating efficiently. Let
$\lambda_{ij}^r$ be the average rates at which type $i$ customers are
routed to server pools $j$. We would like the system state to be such
that $\lambda^r_{ij}$ are close to $\lambda_{ij} r$, where $\{\lambda
_{ij}\}$ is an optimal solution to the following \textit{static planning
problem} (SPP), which is the following linear program:
%
%
\begin{equation}
\label{eqnStaticLP} \min_{\lambda_{ij}, \rho} \rho,
\end{equation}
subject to
%
%
%
%
\begin{eqnarray}
\label{lp-constraint0} \lambda_{ij} &\geq&0\qquad\forall i,j,
\\
\label{lp-constraint1} \sum_j
\lambda_{ij} &=& \lambda_i\qquad\forall i,
\\
\label{lp-constraint2} \sum_i
\lambda_{ij} \big/ (\beta_j \mu_{ij}) &\leq&\rho\qquad
\forall j.
\end{eqnarray}
%

We assume that the SPP has a unique optimal solution
$\{\lambda_{ij}, i\in\CI, j\in\CJ\}, \rho$; and it is such that
the \textit{basic activities}, that is, those pairs, or edges, $(ij)$
for which
$\lambda_{ij}>0$, form a (connected) tree in the graph with vertices
set $\CI\cup\CJ$. The set of basic activities is denoted $\CE$.
These assumptions constitute the \textit{complete resource pooling}
(CRP) condition, which holds ``generically;'' see
\cite{StolyarTezcan}, Theorem
2.2. For a customer type $i$, let $\CS(i) = \{j\dvtx
(ij)\in\CE\}$; for a server type $j$, let \mbox{$\CC(j) = \{i\dvtx
(ij)\in\CE
\}$}.

Note that under the CRP condition, all
(``server pool capacity'') constraints (\ref{lp-constraint2})
are binding; in other words, the optimal solution to SPP
minimizes and ``perfectly balances'' server pool loads.
Optimal dual variables $\nu_i, i\in\CI$ and $\alpha_j, j\in\CJ$,
corresponding to constraints (\ref{lp-constraint1}) and (\ref
{lp-constraint2}),
respectively, are unique and all strictly positive;
$\nu_i$ is interpreted as the ``workload'' associated with
one type $i$ customer,
and $\alpha_j$ is interpreted as the (scaled by $1/r$) maximum rate
at which server pool $j$ can process workload.
The following relations hold:
\begin{eqnarray*}
\alpha_j &=& \max_i \nu_i
\beta_j \mu_{ij},\qquad\nu_i = \min_j
\alpha_j / (\beta_j \mu_{ij}),
\\
\sum_j \alpha_j &=& 1,\qquad\sum
_i \lambda_i \nu_i = \rho\sum
_j \alpha_j = \rho.
\end{eqnarray*}

%

If $\rho< 1$, the system is called \textit{underloaded}; if $\rho= 1$,
the system is called \textit{critically loaded}.
In this paper we consider both cases.

In this paper, we assume that the basic activity tree is known in
advance, and restrict our attention to the basic activities only.
Namely, we assume that a type $i$ customer
service in pool $j$ is allowed only if $(ij)\in\CE$. [Equivalently,
we can a priori assume that $\CE$ is the set of \textit{all} possible
activities, i.e., $\mu_{ij} = 0$ when $(ij)\notin\CE$, and $\CE$
is a tree. In this case CRP requires that all feasible activities are
basic.]

Let $\psi_{ij}^* = \lambda_{ij} / \mu_{ij}$.
Continuing our interpretation of the optimal operating point of the
system, let $\Psi^r_{ij}(t)$ be the number of servers of type $j$
serving customers of type $i$ at time $t$.
It is desirable to have $\Psi^r_{ij}(t) = r \psi^*_{ij} + o(r)$.
Later on we will be also interested in the question of whether or not
the $o(r)$ term can in fact be $O(\sqrt{r})$.

\subsection{Longest-Queue, Freest-Server load balancing algorithm (LQFS-LB)}
\label{sectionLQFS-LB}

For the rest of the paper, we analyze the performance of the following
intuitive load balancing algorithm.

We introduce the following notation (for the system with scaling
parameter $r$):\vspace*{9pt}

$\Psi^r_{ij}(t)$ the number of servers of type $j$ serving
customers of type $i$ at time~$t$;\vspace*{2pt}

$\Psi^r_j(t) = \sum_i \Psi^r_{ij}(t)$ the total number of
busy servers of type $j$ at time $t$;\vspace*{2pt}

$\Psi^r_i(t) = \sum_j \Psi^r_{ij}(t)$ the total number of
servers serving
type $i$ customers at time~$t$;\vspace*{2pt}

$\Xi^r_j(t) = \Psi^r_j(t) / \beta_j$
the instantaneous load of server pool $j$ at time $t$;\vspace*{2pt}

$Q^r_i(t)$ the number of customers of type $i$ waiting for
service at time $t$;\vspace*{2pt}

$X^r_i(t)=\Psi^r_i(t)+Q^r_i(t)$ the total number of customers
of type $i$ in the system
at time $t$.\vspace*{9pt}

The algorithm consists of two parts: routing and scheduling.
``Routing'' determines where an arriving customer goes if it sees
available servers of several different types. ``Scheduling'' determines
which waiting customer a server picks if it sees customers of several
different types waiting in queue.

\textit{Routing}: If an arriving customer of type $i$ sees any unoccupied
servers in server classes in $\CS(i)$, it will pick a server in the
least loaded server pool, that is, $j \in\arg\min_{j\in\CS(i)} \Xi
^r_j(t)$. (Ties are broken in an arbitrary Markovian manner.)\vspace*{1pt}

\textit{Scheduling}: If a server of type $j$, upon completing a service,
sees a customer of a class in $\CC(j)$ in queue, it will pick the
customer from the longest queue, that is, $i \in\arg\max_{j\in\CC
(j)} Q^r_i$. (Ties are broken in an arbitrary Markovian manner.)

By~\cite{GurvichWhitt}, Remark 2.3, the LQFS-LB algorithm described
here is a special case of the algorithm proposed by Gurvich and Whitt,
with constant probabilities $p_i = \frac1I$ (queues ``should'' be
equal), $v_j = \frac{\beta_j}{\sum\beta_j}$ (the proportion of idle
servers ``should'' be the same in all server pools).

\subsection{Basic notation}

Vector $(\xi_i, i\in\CI)$, where $\xi$ can be any symbol, is often
written as $(\xi_i)$ or $\xi_{\CI}$; similarly, $(\xi_j, i\in\CJ
)=(\xi_j)=\xi_{\CJ}$ and $(\xi_{ij}, (ij)\in\CE)=(\xi_{ij})=\xi
_{\CE}$.
We will treat $(\xi_{ij}) = \xi_{\CE}$ as a vector, even
though its elements have two indices.
Unless specified otherwise, $\sum_i \xi_{ij} = \sum_{i \in\CC(j)}
\xi_{ij}$ and $\sum_j \xi_{ij} = \sum_{j \in\CS(i)} \xi_{ij}$.
For functions (or random processes) $(\xi(t), t\ge0)$ we often write
$\xi(\cdot)$. (And similarly for functions with domain different from
$[0,\infty)$.) So, for example, $(\xi_i(\cdot))$ and $\xi_{\CI
}(\cdot)$ both signify $((\xi_i(t), i\in\CI), t\ge0)$. The
indicator function of a set $A$ is denoted $\one_A$; that is, $\one
_A(\omega) = 1$ if $\omega\in A$ and 0 otherwise.

The symbol $\implies$
denotes convergence in distribution of either random variables in the
Euclidean space $\mathbb{R}^d$ (with appropriate dimension $d$), or
random processes in the Skorohod space $D^d[\eta,\infty)$ of RCLL
(right-continuous with left limits) functions on $[\eta,\infty)$, for
some constant $\eta\ge0$. (Unless explicitly specified otherwise,
$\eta=0$.) The symbol $\weakto$ denotes the weak convergence of
probability measures on $\BR^d$, or its one-point compactification
$\overline{\BR}{}^d=\BR^d \cup\{*\}$, where $*$ is the ``point at
infinity.'' We always consider the Borel $\sigma$-algebras on $\BR^d$
and $\overline{\BR}{}^d$.

Standard Euclidean norm of a vector $x\in\mathbb{R}^d$ is denoted $|x|$.
The symbol $\to$ denotes ordinary convergence in $\BR^d$ or
$\overline{\BR}{}^d$.
Abbreviation \textit{u.o.c.} means
\textit{uniform on compact sets} convergence of functions, with the
argument (usually in
$[0,\infty)$) which is clear from the context; \textit{w.p.1} means
convergence \textit{with probability 1}; $\dot{f}(t)$ means $(d/dt)f(t)$.
Transposition of a matrix $H$ is denoted~$H^{\dagger}$; in matrix
expressions vectors are understood as column-vectors.

\section{Fluid model}\label{sectionfluidmodel}

\subsection{Definition}
\label{sectionfluidmodel-def}

We now consider the behavior of fluid models associated with this system.
A fluid model is a set of trajectories that w.p.1 contains any limit
of fluid-scaled trajectories of the original stochastic system.
(We postpone proving this relationship between the fluid models and
fluid limits
until Section~\ref{sectionfluidlimit}, in order to not interrupt the
main content of Section~\ref{sectionfluidmodel};
for now, we just formally define fluid models.)

The term \textit{fluid model} denotes a set of Lipschitz continuous functions
\[
\bigl\{\bigl(a_i(\cdot)\bigr),
\bigl(x_i(\cdot)\bigr), \bigl(q_i(\cdot)\bigr), \bigl(
\psi_{ij}(\cdot)\bigr), \bigl(\rho_j(\cdot)\bigr)\bigr\},
\]
which satisfy the equations below.
[Here $a_i(\cdot)=(a_i(t), t\ge0)$, and similarly for other
components.]
The last two equations involving derivatives are to be satisfied at all
regular points $t$, when the derivatives in question exist. The
interpretation of the components is as follows:
$a_i(t)$ is the total number (actually, ``amount,'' i.e., the number,
scaled by $1/r$) of arrivals of type $i$ customers into the system by
time~$t$, $x_i(t)$ is the number
(``amount'') of customers of type $i$ in the system at time $t$,
$q_i(t)$ is the number
(``amount'') of customers of type $i$ waiting in queue at time $t$,
$\psi_{ij}(t)$ is the number
(``amount'') of customers of type $i$ being served by servers of type
$j$ at time $t$, and $\rho_j(t)$ is the instantaneous load [proportion
of busy servers,
the limit of $\Xi^r_j(t)/r$] in server pool $j$.
%
%
\begin{subequation}\label{eqnfluidmodel}
%
%
\begin{eqnarray}
a_i(t) &=& \lambda_i t\qquad\forall i \in\CI,
\\
%
%
x_i(t) &=& q_i(t) + \sum_j
\psi_{ij}(t)\qquad\forall i \in\CI,
\\
%
%
\label{eqn5c} x_i(t) &=& x_i(0) + a_i(t) -
\sum_j \int_0^t
\mu_{ij} \psi_{ij}(s) \,ds\qquad\forall i \in\CI,
\\
%
%
\rho_j(t) &=& \frac1{\beta_j} \sum
_i \psi_{ij}(t)\qquad\forall j \in\CJ,\\
%
%
\label{eqnworkconservation} \rho_j(t) &=& 1 \mbox{ if
$q_i(t) > 0$ for any $i \in\CC(j)$}\qquad\forall j \in\CJ.
\end{eqnarray}
For any set of server types $\CJ^* \subseteq\CJ$ and
any set of customer types $\CI^* \subseteq\CI$ such that $q_i(t)>0$
for all $i\in\CI^*$, and $q_i(t) > q_{i'}(t)$ whenever $i \in\CI^*$,
$i' \notin\CI^*$ and $\CS(i) \cap\CS(i') \cap\CJ^* \neq
\varnothing$,
%
%
\begin{subequation}\label{eqnfluidalgorithm}
%
%
\begin{eqnarray}
\label{eqnfluidalgorithmqueues}\qquad
&&
\sum_{i \in\CI^*} \sum
_{j \in\CS(i) \cap\CJ^*} \dot{\psi}_{ij}(t)
\nonumber\\[-8pt]\\[-8pt]
&&\qquad=\sum_{j \in\bigcup_{i \in\CI^*} \CS(i) \cap\CJ^*}
\sum_{i' \in
\CC(j)}
\mu_{i'j} \psi_{i'j}(t) - \sum_{i \in\CI^*}
\sum_{j \in
\CS(i) \cap\CJ^*} \mu_{ij} \psi_{ij}(t).
\nonumber
\end{eqnarray}
For any sets of customer types $\CI_* \subseteq\CI$,
and any set of server types $\CJ_* \subseteq\CJ$ such that $\rho
_j(t)<1$ for all $j \in\CJ_*$, and $\rho_j(t) < \rho_{j'}(t)$
whenever $j \in\CJ_*$, $j' \notin\CJ_*$, and $\CC(j) \cap\CC
(j') \cap\CI_* \neq\varnothing$,
%
%
\begin{equation}
\label{eqnfluidalgorithmload}\qquad\sum_{j \in\CJ_*} \sum
_{i \in\CC(j) \cap\CI_*} \dot{\psi}_{ij}(t) = \sum
_{i \in\bigcup_{j \in\CJ_*} \CC(j) \cap\CI_*} \lambda_i - \sum
_{j \in\CJ_*} \sum_{i \in\CC(j) \cap\CI_*}
\mu_{ij} \psi_{ij}(t).
\end{equation}
\end{subequation}
\end{subequation}
The meaning of (\ref{eqnfluidalgorithmqueues}) is as follows.
Consider a set of server types $\CJ^*$. If a set of customer types
$\CI^*$ consists of the ``longest queues for $\CJ^*$'' (we will make
this more precise), then servers in pools $j^* \in\CJ^*$, whenever
they finish serving some customer, will immediately replace her with
someone from a queue in $\CI^*$. In this case, the total number of
customers of types $\CI^*$ in service by servers of types $\CJ^*$
will be increasing at the total rate of servicing all customers by
servers in $\CJ^*$, less the rate of servicing customers of types $\CI
^*$ by servers in $\CJ^*$. The requirements that $\CI^*$ needs to
satisfy for this to be the case are, that there be no customer types
outside $\CI^*$ with longer queues that servers in $\CJ^*$ can serve.
For example, a one-element set $\CI^* = \{i^*\}$ is a valid choice for
a one-element set $\CJ^* = \{j^*\}$ if and only if the customer type
$i^* \in\CC(j^*)$ has the (strictly) longest queue among all of the
customer types that can be served by $j^*$.

The second equation, (\ref{eqnfluidalgorithmload}), describes the
fact that if a set of server pools $\CJ_*$ consists of the ``least
loaded server pools available to $\CI_*$,'' then servers in pools $j^*
\in\CJ^*$, whenever they finish serving some customer, will
immediately replace her with someone from queue $i^*$. For example, a
one-element set $\CJ_* = \{j_*\}$ is a valid choice for a one-element
set $\CI_* = \{i_*\}$ if and only if the server pool $j_* \in\CS
(i_*)$ has the (strictly) smallest load $\rho_{j_*}$ among all of the
server pools that can serve $i_*$.

\subsection{Behavior in the vicinity of equilibrium point}

We define the \textit{equilibrium} (\textit{invariant}) point of the
underloaded
($\rho< 1$) fluid model to be the state $\psi_{ij} = \psi^*_{ij}$
and $q_i = q=0$ for all $i \in\CI$, $j \in\CJ$.
[All other components of the fluid model are also constant and uniquely
defined by $(\psi^*_{ij})$ and $q$.] Clearly,\vspace*{1pt} $\psi
_{ij}(t) \equiv
\psi^*_{ij}$ and $q_i(t) \equiv q$ is indeed a stationary fluid model.
Desirable system behavior would be to have $(\psi_{ij}(t)) \to(\psi
^*_{ij})$ as $t \to\infty$.\vspace*{1pt}

Note that if the initial system state is in the vicinity of the
equilibrium point (with $\rho< 1$), then there is no queueing in the
system, and we can describe the system with just the variables $(\psi
_{ij}(t))$. This will be true for at least some time (depending on
$\rho$ and the initial distance to the equilibrium point), because the
fluid model is Lipschitz.

The following is a ``state space collapse'' result for the underloaded
fluid model in the neighborhood of the equilibrium point.
%
%
\begin{theorem}\label{thmunderload,linearequation}
Let $\rho< 1$. There exists a sufficiently small $\varepsilon> 0$,
depending only on the system parameters, such that for all sufficiently
small $\delta$ the following holds. There exist $T_1 = T_1(\delta)$
and $T_2 = T_2(\delta)$, $0 < T_1 < T_2$, such that if the initial
system state $(\psi_{ij}(0))$ satisfies
\[
\bigl\llvert\bigl(\psi_{ij}(0)\bigr) - \bigl(\psi^*_{ij}
\bigr)\bigr\rrvert< \delta, 
\]
then for all $t \in[T_1, T_2]$ the system state satisfies
\[
\bigl\llvert\bigl(\psi_{ij}(t)\bigr) - \bigl(\psi^*_{ij}
\bigr)\bigr\rrvert< \varepsilon,\qquad
\rho_j(t) =
\rho_{j'}(t) \qquad\mbox{for all $j, j' \in\CJ$}.
\]
Moreover, $T_1 \downarrow0$ and $T_2 \uparrow\infty$ as $\delta
\downarrow0$. The evolution of the system on $[T_1, T_2]$ is described
by a linear ODE, specified below by (\ref{eqnunderloadODEmatrix}).
\end{theorem}

If the fluid system is critically loaded ($\rho= 1$), it may have
queues at equilibrium, and the equilibrium is nonunique.
Namely, the definition of an equilibrium (invariant) point for $\rho
=1$ is the same as for the underloaded system, except the condition on
the queues becomes $q_i= q$ for some constant $q\ge0$. In the next
Theorem~\ref{thmcriticalload,linearequation} we will only consider
the case of positive queues ($q>0$) for the critically loaded fluid model.
%
%
\begin{theorem}\label{thmcriticalload,linearequation}
Let $\rho= 1$, and consider an equilibrium point with $q>0$.
There exists a sufficiently small $\varepsilon> 0$, depending only on the
system parameters, such that for all sufficiently small $\delta> 0$
the following holds. There exist $T_1 = T_1(\delta)$ and $T_2 =
T_2(\delta)$, $0 < T_1 < T_2$, such that if the initial system state satisfies
\[
\bigl\llvert\bigl(\psi_{ij}(0)\bigr) - \bigl(\psi^*_{ij}
\bigr)\bigr\rrvert< \delta,\qquad\bigl\llvert q_i(0) - q\bigr
\rrvert<
\delta\qquad\mbox{for all $i \in\CI$},
\]
then for all $t \in[T_1, T_2]$ the system state satisfies
\begin{eqnarray*}
\bigl\llvert\bigl(\psi_{ij}(t)\bigr) - \bigl(\psi^*_{ij}
\bigr)\bigr\rrvert&<& \varepsilon,\qquad\bigl\llvert q_i(t) -
q\bigr\rrvert<
\varepsilon\qquad\mbox{for all $i \in\CI$},
\\
q_i(t) &=& q_{i'}(t) \qquad\mbox{for all $i, i'
\in\CI$}.
\end{eqnarray*}
Moreover, $T_1 \downarrow0$ and $T_2 \uparrow\infty$ as $\delta
\downarrow0$. The evolution of the system on $[T_1, T_2]$ is described
by a linear ODE specified below by (\ref{eqncriticalloadODE}).
\end{theorem}


In the rest of this section and the paper, the values associated with
a stationary fluid model, ``sitting'' at an equilibrium point,
are referred to as \textit{nominal}. For example, $\psi_{ij}^*$ is the
nominal occupancy (of pool $j$ by
type~$i$), $\lambda_{i}$ is the nominal arrival rate, $\lambda_{ij}$
is the nominal
routing rate [along activity $(ij)$], $\psi^*_{ij} \mu_{ij}=\lambda
_{ij}$ is the nominal service rate
(of type $i$ in pool $j$), $\sum_j \psi^*_{ij} \mu_{ij}=\lambda_{i}$
is the nominal total service rate (of type $i$), $\rho$ is the nominal
total occupancy
(of each pool $j$), etc.
\begin{pf*}{Proof of Theorem~\ref{thmunderload,linearequation}}
Let us choose a suitably small $\varepsilon>0$ (we will specify how small
later). Because the fluid model trajectories are continuous, we can
always choose some $T_2>0$ such that, for all sufficiently small
$\delta>0$, if $\llvert(\psi_{ij}(0)) - (\psi^*_{ij})\rrvert<
\delta$,
then $\llvert(\psi_{ij}(t)) - (\psi^*_{ij})\rrvert<
\varepsilon$ for all $t
\leq T_2$. We will show that $\rho_j(t) = \rho_{j'}(t)$ for all $j,
j' \in\CJ$, in $[T_1,T_2]$ for some $T_1$ depending on $\delta$.

Consider $\rho_*(t) = \min_j \rho_j(t)$, $\rho^*(t) = \max_j \rho_j(t)$
and assume $\rho_*(t)<\rho^*(t)$. Let $\CJ_*(t) = \{j\dvtx\rho
_j(t) =
\rho
_*(t)\}$. As long as $\rho_*(t) < \rho^*(t)$, $\CJ_*(t)$ is of
course a
strict subset of $\CJ$. The total\vspace*{1pt} arrival rate
to servers of type $j\in\CJ_*(t)$ is
$\sum_{i \in\bigcup_{j \in\CJ_*(t)} \CC(j)} \lambda_i$.
By the assumption of the connectedness of the basic activity tree, this
is strictly greater (by a constant) than the nominal arrival rate $\sum
_{i \in\CC(j), j \in\CJ_*(t)} \lambda_{ij}$. The total rate of
departures from those servers is $\sum_{i \in\CC(j), j \in\CJ
_*(t)} \mu_{ij} \psi_{ij}(t)$. For small $\varepsilon$, the assumption
$\llvert(\psi_{ij}(t)) - (\psi^*_{ij})\rrvert< \varepsilon
$ implies that this
is close to the nominal departure rate, so the arrival rate exceeds the
service rate by at least a constant. (This determines what ``suitably
small'' means for $\varepsilon$ in terms of the system parameters.)
Consequently, as long as $\rho_*(t) < \rho^*(t)$, the minimal load
$\rho_*(t)$ is increasing at a rate bounded below by a constant.
Similarly, as long as $\rho_*(t) < \rho^*(t)$, the maximal load $\rho
^*(t)$ is decreasing at a rate bounded below by a constant. Therefore,
the difference $\rho^*(t) - \rho_*(t)$ is decreasing at a rate
bounded below by a constant whenever it is positive. Thus, in finite
time $T_1 = T_1(\delta)$ we will arrive at a state $\rho_*(t) = \rho
^*(t)$. [Clearly, $T_1(\delta)\to0$ as $\delta\to0$.] Since the
function $\rho^*(\cdot) - \rho_*(\cdot)$ is Lipschitz (hence
absolutely continuous), bounded below by 0 and (for $t \leq T_2$) has
nonpositive derivative whenever it is differentiable, the condition
$\rho_*(t) = \rho^*(t)$ will continue to hold for $T_1 \leq t \leq T_2$.

It remains to derive the differential equation, and to show
that $T_2$ can be chosen depending on $\delta$ so that
$T_2 \uparrow\infty$ as $\delta\downarrow0$.

Once we are confined to the manifold $\rho_j(t) = \rho_{j'}(t) = \rho
(t)$ for all $t$, the system evolution is determined in terms of only
$I$ independent variables. Decreasing $\varepsilon$ if necessary to
ensure that there is no queueing while $\llvert(\psi_{ij}(t)) -
(\psi^*_{ij})\rrvert< \varepsilon$,
we can take the $I$ variables to be $\psi_i(t):= \sum_j \psi
_{ij}(t)$. Given
$(\psi_i(t))$ we know $\rho(t)$ as $(\sum_i \psi_i(t)) / (\sum_j
\beta_j)$. Consequently, we know $\sum_i \psi_{ij}(t) = \rho(t)
\beta_j$ and $\sum_j \psi_{ij}(t) = \psi_i(t)$. On a tree, this
allows us to solve for $\psi_{ij}(t)$; the relationship will clearly
be linear, that is,
%
%
\begin{equation}
\label{eqnpsiijintermsofpsii} \bigl(
\psi_{ij}(t)\bigr) = M \bigl(\psi_i(t)\bigr)
\end{equation}
for some matrix $M$. For future reference, we define the (``load
balancing'') linear mapping $M$ from $y\in\BR^I$ to $z= (z_{ij},
(ij)\in\CE) \in\BR^{I+J-1}$ as follows:
$z=My$ is the unique solution of
%
%
\begin{equation}
\label{M-def} \eta= \frac{\sum_i y_i}{\sum_j \beta_j};\qquad
\sum_i z_{ij} = \eta\beta_j\qquad\forall j;\qquad\sum
_j z_{ij} = y_i\qquad\forall i.
\end{equation}

The evolution of $\psi_i(t)$ is given by
%
%
\begin{equation}
\label{eqnunderloadODE} \dot{\psi}_i(t) = \lambda_i -
\sum_j \mu_{ij} \psi_{ij}(t)\qquad
\forall i.
\end{equation}
[This follows from (\ref{eqn5c}) and the fact that $q_i(t)=0$.]
Then, by the above arguments we see that this entails (in matrix form)
%
%
\begin{equation}
\label{eqnunderloadODEmatrix} \bigl(\dot{\psi}_i(t)\bigr) = (
\lambda_i) + A_{u} \bigl(\psi_i(t)\bigr),
\end{equation}
where $A_u$ is an $I \times I$ matrix, $A_u = GM$. Here, $G$ is a $I
\times(I+J-1)$ matrix with entries $G_{i,(kj)} = -\mu_{ij}$ if $i=k$, and
$G_{i,(kj)} = 0$ otherwise.

It remains to justify the claim that $T_2(\delta) \uparrow\infty$ as
$\delta\downarrow0$. This follows from the fact that, as long as
$t\ge T_1$ and $\llvert(\psi_{ij}(t)) - (\psi^*_{ij})\rrvert<
\varepsilon$,
the evolution of the system is described by the linear ODE above. The
solutions have the general form
\begin{eqnarray*}
\psi_\CI(t) - \psi^*_\CI&=& \exp\bigl(A_u(t-T_1)
\bigr) \bigl(\psi_\CI(T_1) - \psi^*_\CI
\bigr), \\
\psi_\CE(t) - \psi^*_\CE&=& M\bigl(
\psi_\CI(t) - \psi^*_\CI\bigr),
\end{eqnarray*}
where $M$ and $A_u$ are constant matrices depending on the system
parameters. Therefore, if $\llvert\psi_\CI(T_1) - \psi^*_\CI
\rrvert\leq
\delta$ is sufficiently small, then the time it takes for $\psi_\CE
(t)$ to escape the set $\llvert\psi_\CE(t) - \psi^*_\CE\rrvert<
\varepsilon
$ can be made arbitrarily large. Since as $\delta\downarrow0$ we have
$T_1(\delta) \downarrow0$, and the system trajectory is Lipschitz,
taking $\llvert\psi_\CE(0) - \psi^*_\CE\rrvert<
\delta$ for small enough
$\delta$ will guarantee that $\llvert\psi_\CI(T_1) - \psi^*_\CI
\rrvert$
is small, and hence we can choose $T_2(\delta) \uparrow\infty$.
\end{pf*}

The proof of Theorem~\ref{thmcriticalload,linearequation} proceeds
similarly; we outline only the differences.
\begin{pf*}{Proof of Theorem~\ref{thmcriticalload,linearequation}}
First, since we assume that $\varepsilon>0$ is sufficiently small
and $\llvert q_i(t) - q\rrvert< \varepsilon$, $i\in\CI$,
for all $t \leq T_2$,
we clearly have $\rho_j(t) = 1$, $j\in\CJ$, for all $t \leq T_2$.
The equality of queue lengths in $[T_1,T_2]$ is shown analogously
to the proof of $\rho_*(t)=\rho^*(t)$ for in the underloaded case.
Namely, the smallest queue must increase and the largest queue must decrease
[as long as not all $q_i(t)$ are equal], because it is getting less
(resp., more) service than nominal [we choose $\varepsilon$ small enough
for this to be true provided $\llvert(\psi_{ij}(t)) - (\psi
^*_{ij})\rrvert<
\varepsilon$]. Thus, in $[T_1, T_2]$ we will have $q_i(t) = q_{i'}(t)$
for all $i, i' \in\CI$.

The linear equation is modified as follows. We have
\[
\dot{x}_i(t) = \lambda_i - \sum
_j \mu_{ij} \psi_{ij}(t),
\]
where $x_i(t) = \psi_i(t) + q_i(t)$. Since we know that all $q_i(t)$
are equal and positive, we have $q_i(t) = q(t) = \frac1I (\sum x_k(t)
- \sum\beta_j)$, and therefore
\[
\dot{\psi}_i(t) = \dot{x}_i(t) - \frac1I \sum
_k \dot{x}_k(t).
\]
The rest of the arguments proceed as above to give
%
%
\begin{equation}
\label{eqncriticalloadODE} \bigl(\dot{\psi}_i(t)\bigr) = \biggl(
\lambda_i - \frac1I \sum_i
\lambda_i\biggr) + A_{c} \bigl(\psi_i(t)
\bigr)
\end{equation}
for the appropriate matrix $A_c$ which can be computed explicitly from
the basic activity tree. (Of course, in $[T_1,T_2]$, the trajectory
$(\psi_{ij}(\cdot))$
uniquely determines
$(\psi_{i}(\cdot))$, $(x_{i}(\cdot))$ and $(q_{i}(\cdot))$.)

Just as above, the existence of the linear ODE, together with the fact
that $T_1 \downarrow0$ as $\delta\downarrow0$, implies that $T_2
\uparrow\infty$ as $\delta\downarrow0$.
\end{pf*}

To compute the matrix $M$, and therefore the matrices $A_u$ and $A_c$,
we will find the following observation useful. If $(\psi_{ij}(t))=M
(\psi_{i}(t))$, then the common value
$\rho(t)=\rho_j(t), \forall j$, is
\[
\rho(t) = \sum_i \psi_i(t) \Big/ \sum
_j \beta_j.
\]
This allows us to find the values $(\psi_{ij}(t))$ from $(\psi_i(t))$
as follows: if $i$ is a customer-type leaf, then $\psi_{ij}(t) = \psi
_i(t)$; if $j$ is a server-type leaf, then $\psi_{ij}(t) = \rho(t)
\beta_j$; we now remove the leaf and continue with the smaller tree.
Inductively, for an activity $i_0 j_0$ we find
%
%
\begin{eqnarray}
\label{eqnpsiij}\quad
\psi_{i_0 j_0}(t) &=& \sum
_{i \preceq(i_0,j_0)} \psi_{i}(t) - \sum
_{j \preceq(i_0,j_0)} \rho(t)\beta_j
\nonumber\\[-8pt]\\[-8pt]
&=&\frac{1}{\sum\beta_j} \biggl(\sum_{i \preceq(i_0,j_0)} \sum
_{j
\preceq(j_0,i_0)} \psi_i(t) \beta_j - \sum
_{i \preceq(j_0,i_0)} \sum_{j \preceq(i_0,j_0)}
\psi_i(t) \beta_j \biggr).
\nonumber
\end{eqnarray}
Here, the relation $\preceq$ is defined as follows.
Suppose we disconnect the basic activity tree by removing the edge
$(i_0,j_0)$. Then for any node $k$ (either customer type or server
type), we say $k \preceq(i_0,j_0)$ if it falls in the same component
as $i_0$; otherwise, $k \preceq(j_0,i_0)$.

%
%
\begin{figure}

\includegraphics{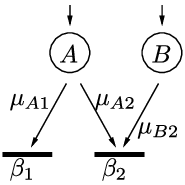}

\caption{Example for calculation of the matrix $M$.}
\label{figAuexample}
\end{figure}

For example, consider the network in Figure~\ref{figAuexample}.
For it, we obtain
\[
\pmatrix{ \psi_{A1}
\cr
\psi_{A2}
\cr
\psi_{B2} } %
= %
\pmatrix{ \displaystyle \frac{\beta_1}{\beta_1 + \beta_2} &
\displaystyle \frac{\beta_1}{\beta_1 + \beta_2} \vspace*{2pt}\cr \displaystyle 1 -
\frac{\beta_1}{\beta_1 + \beta_2} & \displaystyle
-\frac{\beta_1}{\beta_1 + \beta_2} \vspace*{2pt}\cr 0 & 1 }
\pmatrix{ \psi_A
\cr
\psi_B }.
\]

Since in the underload we have
\[
\dot{\psi}_i(t) = \lambda_i - \sum
_j \mu_{ij}\psi_{ij}(t),
\]
%
we obtain an expression for $A_u$, given in Lemma~\ref{lemmaentriesofAu}(i)
just below.
%
%
\begin{lemma}\label{lemmaentriesofAu}
\textup{(i)} The entries $(A_u)_{ii'}$ of the matrix $A_u$ (for the underload
case, $\rho<1$)
are as follows.
The coefficient of $\psi_i$ in $\dot\psi_i$ is
\[
(A_u)_{ii} = -\frac{1}{\sum_j \beta_j} \sum
_{j \in\CS(i)} \mu_{ij} \sum_{j' \preceq(j,i)}
\beta_{j'}.
\]
The coefficient of $\psi_{i'}$ in $\dot\psi_i$ is
\begin{eqnarray*}
(A_u)_{ii'} &=& \frac{1}{\sum_j \beta_j} \biggl[ -\sum
_{j \in\CS
(i), j \neq j_{ii'}} \mu_{ij} \sum_{j' \preceq(j,i)}
\beta_{j'} + \mu_{ij_{ii'}} \sum_{j' \preceq(i,j_{ii'})}
\beta_{j'} \biggr]
\\
&=& (A_u)_{ii} + \mu_{ij_{ii'}}.
\end{eqnarray*}
Here, $j_{ii'} \in\CS(i)$ is the neighbor of $i$ such that, after
removing the edge $(i,j_{ii'})$ from the basic activity tree, nodes $i$
and $i'$ will be in different connected components. (Such a node is
unique, since there is a unique path along the tree from $i$ to $i'$.)

\mbox{}\hphantom{\textup{i}}\textup{(ii)} The matrix $A_u$ is nonsingular.

\textup{(iii)} The matrix $A_u$ depends only on $(\beta_{j})$, $(\mu_{ij})$ and
the basic activity tree structure $\CE$, and does \textup{not}
depend on $(\lambda_{i})$ and $(\psi^*_{ij})$.
\end{lemma}
\begin{pf}
(i) In the proof of Theorem~\ref{thmunderload,linearequation} we
showed $A_u = GM$, where $G$ is a $I \times(I+J-1)$ matrix with
entries $G_{i,(kj)} = -\delta_{ik}\mu_{ij}$,
where $\delta_{ik}$ is the Kronecker's delta function
and $M$ is the $I \times(I+J-1)$ load-balancing matrix whose entries
are determined from the expression (\ref{eqnpsiij}). The form of the
entries for $A_u$ now follows. The equality between the two expressions
for the off-diagonal entries is a consequence of the fact that, for all
$j'$, exactly one of $j' \preceq(ij)$, $j' \preceq(ji)$ holds.

(ii) In the case $\rho<1$,
in the vicinity of the equilibrium point, the derivative $(\dot{\psi
}_{i})=(\lambda_i)+A_u (\psi_i)$
(which can be any real-valued $I$-dimensional vector, within a small
neighborhood of the origin)
uniquely determines $(\psi_{ij})$, and then $(\psi_{i})$ as well.
Indeed, we have the system of $I+J$ linear equations $\lambda_i-\sum_j
\mu_{ij} \psi_{ij} =\dot{\psi}_{i},
\forall i$ and
$\sum_i \psi_{ij} = \hat\rho\beta_j, \forall j$, for the $I+J$
variables $\hat\rho, (\psi_{ij})$.
This system has unique solution,
because $\hat\rho$ is uniquely determined by the workload derivative condition
\[
\sum_i \nu_i \dot{
\psi}_i = \sum_i \nu_i
\lambda_i - \sum_j \hat\rho
\alpha_j,
\]
and then the values of $\psi_{ij}$
are determined by sequentially ``eliminating'' leaves of the basic
activity tree.

(iii) Follows from (i).
\end{pf}

%
%
\begin{lemma}\label{lemmaentriesofAc}
\textup{(i)} The entries $(A_c)_{ii'}$ of the matrix $A_c$ (for the critical
load case, $\rho=1$)
are as follows:
%
%
\begin{equation}
\label{eq-ac2} (A_c)_{ii'} = (A_u)_{ii'}
- \frac1I \sum_{k} (A_u)_{ki'}.
\end{equation}

\mbox{}\hphantom{\textup{i}}\textup{(ii)}  The matrix $A_c$ has rank $I-1$. The $(I-1)$-dimensional subspace
$L=\{y | \sum_i y_i =0\}$ is invariant under the transformation
$A_c$, that is, $A_c L \subseteq L$.
Letting $\pi$ denote the matrix of the orthogonal projection
[along $(1,\ldots,1)^{\dagger}$] onto $L$, we have $A_c = \pi A_u$.
Restricted to $L$, the transformation $A_c$ is invertible.

\textup{(iii)} The linear transformation $A_c$, restricted to subspace $L$,
depends only on $(\mu_{ij})$ and
the basic activity tree structure $\CE$, and does \textup{not}
depend on $(\beta_{j})$, $(\lambda_{i})$ and $(\psi^*_{ij})$.
\end{lemma}
\begin{pf}
(i) The fluid model here is such that there are always nonzero queues,
which are equal across customer types.
We can write
%
%
\begin{eqnarray}
\label{eq-ac1} \dot{\psi}_i(t) &=& \dot{x}_i(t) - \frac1I
\sum_k \dot{x}_k(t) \nonumber\\[-8pt]\\[-8pt]
&=& \biggl(
\lambda_i - \sum_j \mu_{ij}
\psi_{ij}(t)\biggr) - \frac1I \sum_k
\biggl(\lambda_k - \sum_j
\mu_{kj}\psi_{kj}(t)\biggr),\nonumber
\end{eqnarray}
which implies (\ref{eq-ac2}).\vadjust{\goodbreak}

\mbox{}\hphantom{i}(ii) First of all, it is not surprising that $A_c$ does not have full
rank: the linear ODE defining $A_c$ is such that $\sum_i \psi_i(t) =
\sum_j \beta_j$ at all times, so there are at most $(I-1)$ degrees of
freedom in the system. Also, it will be readily seen that (\ref
{eq-ac2}) asserts precisely that $A_c = \pi A_u$. Since $A_u$ is
invertible and $\pi$ has rank $I-1$, their composition has rank $I-1$.
Since the image of $A_c$ is contained in $L$, the image of $A_c$ (as a
map from $\BR^I$) must be equal to all of~$L$.

It remains to check that $A_c$ restricted to $L$ still has rank $I-1$.
To see this, we observe that the simple eigenvalue $0$ of $A_c$ has as
its unique right eigenvector the vector $A_u^{-1}
(1,1,\ldots,1)^{\dagger}$. We will be done once we show that this
eigenvector does not belong to $L$. Suppose instead that $A_u v =
(1,1,\ldots,1)^{\dagger}$ for some $v \in L$, $\sum_i v_i = 0$. Then,
for a small $\varepsilon> 0$, the state $\psi^*_\CI- \varepsilon v$ (with
balanced pool loads, all equal to the optimal $\rho$) would be such
that the derivatives of all components $\psi_i$ would be strictly
negative. This is, however, impossible because the total rate at which
the system workload is served must be zero,
\[
\frac{d}{dt} \sum_i \nu_i
\psi_i = \sum_i \nu_i
\lambda_i - \sum_j \rho
\alpha_j = 0.
\]

(iii) The specific expression (\ref{eq-ac2}) for $A_c$ may depend on
the pool sizes $(\beta_j)$. However, $A_c$ is a singular $I\times I$
matrix, and our claim is only about the transformation of the
$(I-1)$-dimensional subspace $L$ that $A_c$ induces; this
transformation does \textit{not} depend on $(\beta_j)$, as the following
argument shows.

Pick any $(ij)\in\CE$. Modify the original system by replacing $\beta
_j$ by $\beta_j + \delta$ and $\lambda_i$ by $\lambda_i + \delta
\mu_{ij}$; this means that the nominal $\psi^*_{ij}$ is replaced by
$\psi^*_{ij}+\delta$. Then, using notation $\gamma_i(t)=\psi
_i(t)-\psi
^*_i$, the linear ODE
%
%
\begin{equation}
\label{eq-ac3} \bigl(\dot{\gamma}_i(t)\bigr)= A \bigl(
\gamma_i(t)\bigr),
\end{equation}
which we obtain from the ODE
(\ref{eq-ac1}) for the original and modified systems, has exactly the
same matrix $A$, which implies $A=A_c$. Thus, the transformation $A_c$
must not depend on $\beta_j$.

An alternative argument is purely analytic. Recall that to compute
$(A_u)_{ij}$ we used (\ref{eqnpsiij}). In critical load, we have
$\rho(t) \equiv1$, so the (left) equation (\ref{eqnpsiij}) for
$\psi_{i_0 j_0}(t)$ simplifies to
%
%
\begin{equation}
\label{eqnpsiijforrho=1} \psi_{i_0 j_0}(t) = \sum
_{i \preceq(i_0, j_0)} \psi_i(t) - \sum
_{j
\preceq(i_0, j_0)} \beta_j.
\end{equation}
If we substitute this in the right-hand side of (\ref{eq-ac1}), we
will obtain a different expression for $\dot\psi_i(t)$. While its
constant term will depend on $\beta_\CJ$, the linear term will not,
since the linear term of (\ref{eqnpsiijforrho=1}) does not depend
on $\beta_\CJ$. That is, we have found a way of writing down a matrix
for $A_c$ which clearly does not depend on the $\beta_\CJ$.
\end{pf}

\subsection{Definition of local stability}
\label{sec-local-stability-def}

We say that the (fluid) system is \textit{locally stable}, if all fluid
models starting in a sufficiently small neighborhood of an equilibrium point
(which is unique for $\rho<1$; and for $\rho=1$ we consider any equilibrium
point with equal queues $q>0$)
are such that, for fixed constant $C>0$,
\[
\bigl|\bigl(\psi_{ij}(t)\bigr) - \bigl(\psi_{ij}^*\bigr)\bigr| \le
\Delta_0 e^{-Ct},
\]
where\vspace*{1pt} $\Delta_0= |(\psi_{ij}(0)) - (\psi_{ij}^*)| + |(q_i(0)) -
(q,\ldots,q)^{\dagger}|$.
Note that in the case $\rho=1$ it is \textit{not}
required that $q_i(t)\to q$, for $q$ associated with the
chosen equilibrium
point.
However, local stability will guarantee convergence of queues
$q_i(t)\to\overline q$, with some $\overline q > 0$ possibly different
from $q$.
Indeed, the exponentially fast convergence $\psi_\CE(t) \to\psi
^*_\CE$
of the occupancies to the nominal, guarantees that for
some fixed constant $C_1>0$, any $i$ and any $s \ge t\ge0$,
\[
\bigl|x_i(s) -x_i(t)\bigr| \le\int_t^s
\biggl|\lambda_i - \sum_j
\mu_{ij} \psi_{ij}(\xi)\biggr| \,d\xi\le C_1
\Delta_0 e^{-Ct}.
\]
Therefore, each $x_i(t)$, and then each $q_i(t)$, also converges
exponentially fast.
Then we can apply Theorem~\ref{thmcriticalload,linearequation} to
show that all $q_i(t)$ must be equal starting some time point;
therefore they converge to the same value $\overline q$, which is such that
that $\llvert\overline q-q\rrvert\leq C_0 \Delta_0$ for some
constant $C_0>0$
depending only on the system parameters. In other words, local
stability guarantees convergence to an equilibrium point not too far
from the ``original'' one. (We omit further detail, which are rather
straightforward.)

By Theorems~\ref{thmunderload,linearequation} and \ref
{thmcriticalload,linearequation} we see that the local stability
is determined by the stability of a linear ODE, which in turn is
governed by the eigenvalues of the matrix $A_u$ or $A_c$.
We will call matrix $A_u$ stable if
all its eigenvalues have negative real part. We call matrix $A_c$
stable if all its eigenvalues have negative real part, except one
simple eigenvalue
$0$.\setcounter{footnote}{1}\footnote{A matrix $A$
with all eigenvalues having negative real part is
usually called \textit{Hurwitz}. So, $A_u$ stability
is equivalent to $A_u$ being
Hurwitz; while $A_c$ stability definition is slightly
different, due to $A_c$ singularity.
A symmetric matrix $A$ is Hurwitz if and only if
it is negative definite, but neither $A_u$ nor $A_c$ is, in general, symmetric.}
In this terminology, \textit{the local stability of the system is
equivalent to
the stability of
the matrix $A$ in question} (\textit{either $A_u$ or $A_c$}).
On the other hand, if $A$ has an eigenvalue with positive real part,
the ODE has solutions diverging from equilibrium
$(\psi_i^*)$ exponentially fast; if $A$ has
(a pair of conjugate) pure imaginary
eigenvalues,
the ODE has oscillating, never converging solutions.

\subsection{Fluid model as a fluid limit}\label{sectionfluidlimit}

In this section we show that the set of fluid models defined
in Section~\ref{sectionfluidmodel-def}
contains (in the sense
specified shortly) all possible limits of ``fluid scaled'' processes.
We consider a sequence of systems indexed by~$r$, with
the input rates being $\lambda^r_i = r \lambda_i + o(r)$, server pool
sizes being
$\beta_j r$ and the service rates $\mu_{ij}$ unchanged with $r$.
Recall the notation in Section~\ref{sectionLQFS-LB}. We also add the
following notation:\vspace*{9pt}

$A^r_i(t)$ the number of customers of type $i$ who have entered
the system by time $t$ (a Poisson process of rate $\lambda^r_i$);

$S^r_{ij}(t)$ the number of customers of type $i$ who have been
served by servers of type $j$ if a total time $rt$ has been spent on
these services (a Poisson process of rate $\mu_{ij} r$).\vspace*{9pt}

Let\vspace*{2pt} $\Pi^{(a)}_i(\cdot)$, $i\in\CI$, and $\Pi^{(s)}_{ij}(\cdot)$,
$(ij)\in\CE$,
be independent unit-rate Poisson processes. We can assume that, for
each $r$,
\[
A^r_i(t) = \Pi^{(a)}_i\bigl(
\lambda^r_i t\bigr),\qquad S^r_{ij}(t)
= \Pi^{(s)}_{ij}(\mu_{ij} rt).
\]
Then, by the functional strong law of large numbers, with probability 1,
uniformly on compact subsets of $[0,\infty)$,
%
%
\begin{equation}
\label{eqnassumptionsonA,S} \frac1r A^r_i(t) \to
\lambda_i t, \qquad\frac1r S^r_{ij}(t) \to
\mu_{ij} t. 
\end{equation}


We consider the following scaled processes:
\begin{eqnarray*}
x_i^r(t) &=& \frac1r X^r_i(t),\qquad
q_i^r(t) = \frac1r Q^r_i(t),\qquad
\psi^r_{ij}(t) = \frac1r \Psi^r_{ij}(t),
\\
\rho^r_j(t) &=& \frac1r \Xi^r_j(t),\qquad
a^r_i(t) = \frac1r A^r_i(t).
\end{eqnarray*}

%
%
\begin{theorem}
Suppose
\[
\bigl\{\bigl(x^r_i(0)\bigr), \bigl(q^r_i(0)
\bigr), \bigl(\psi^r_{ij}(0)\bigr), \bigl(
\rho^r_j(0)\bigr)\bigr\} \to\bigl\{\bigl(x_i(0)
\bigr), \bigl(q_i(0)\bigr), \bigl(\psi_{ij}(0)\bigr),
\bigl(\rho_j(0)\bigr)\bigr\}.
\]
Then w.p.1 any subsequence of $\{r\}$
contains a further subsequence along which u.o.c.,
\begin{eqnarray*}
&&\bigl\{\bigl(a^r_i(\cdot)\bigr), \bigl(x^r_i(
\cdot)\bigr), \bigl(q^r_i(\cdot)\bigr), \bigl(
\psi^r_{ij}(\cdot)\bigr), \bigl(\rho^r_j(
\cdot)\bigr)\bigr\}
\\
&&\qquad\to\bigl\{\bigl(a_i(\cdot)\bigr), \bigl(x_i(\cdot)\bigr),
\bigl(q_i(\cdot)\bigr), \bigl(\psi_{ij}(\cdot)\bigr),
\bigl(\rho_j(\cdot)\bigr)\bigr\},
\end{eqnarray*}
where the limiting trajectory (on the right-hand side) is a fluid model.
\end{theorem}
\begin{pf}
Given property (\ref{eqnassumptionsonA,S}),
the probability $1$, u.o.c., convergence
along a subsequence
to a Lipschitz continuous set of functions
easily follows.
The only nontrivial properties of a fluid model that need to be
verified for the limit
are (\ref{eqnfluidalgorithm}). Let us consider a regular time $t$:
namely, such that all the components of a limit trajectory have derivatives,
and moreover the minimums and maximums over any subset of components
have derivatives as well.
Consider a sufficiently small interval $[t,t+\Delta t]$,
and consider the behavior of the (fluid-scaled) pre-limit trajectory
in this interval. Then, it is easy to check that the
conditions (\ref{eqnfluidalgorithm}) on the derivatives must hold;
the argument here is very standard---we omit details.
\end{pf}

\section{Special cases in which fluid models are stable}
\label{sectionspecialcases}

In this section we analyze two special cases
of the system parameters, for which we demonstrate
convergence results.
In Section~\ref{secstabilityformuij=muj} we consider the case
when there exists a set of positive $\mu_j$, $j\in\CJ$,
such that $\mu_{ij} = \mu_j$ for $(ij)\in\CE$
[i.e., the service rate $\mu_{ij}$ is constant
across all $i\in\CC(j)$]; we show global convergence of fluid models
to equilibrium.
In Section~\ref{seclocalstabilityformuij=mui} we consider
the case when there exists a set of positive $\mu_i$, $i\in\CI$,
such that $\mu_{ij} = \mu_i$ for $(ij)\in\CE$
[i.e., the service rate $\mu_{ij}$ is constant
across all $j\in\CS(i)$]; we show local stability of the fluid model
(i.e., stability of $A_u$ and $A_c$).

\subsection{\texorpdfstring{Global stability in the case $\mu_{ij} = \mu_j$, $(ij)\in\CE$}
{Global stability in the case mu ij = mu j, (ij) in E}}
\label{secstabilityformuij=muj}

We call the system \textit{globally stable} if
any fluid model, with arbitrary initial state,
converges to an equilibrium point as $t\to\infty$.
[This of course implies $\rho_j(t) \to\rho$ for all $j \in\CJ$ and
$\psi_{ij}(t) \to\psi^*_{ij}$ for all $i \in\CI$, $j \in\CJ$.
Note that, in the underload, the definition
necessarily implies $q_i(t) \to0$ for all $i \in\CI$, while in the
critical load
it requires $q_i(t) \to q$ for all $i \in\CI$ and some $q \geq0$.]
%
%
\begin{theorem}\label{thmglobalstabilityformuij=muj}
The system with $\mu_{ij} = \mu_j$, $(ij)\in\CE$, is globally
stable both for $\rho< 1$ and for $\rho= 1$. In addition, the system
is locally stable as well (i.e.,
the matrices $A_u$ and $A_c$ are stable).
\end{theorem}
\begin{pf}
Consider the underloaded system, $\rho< 1$, first. First, we show that
the lowest load cannot stay too low. Suppose the minimal load $\rho
_*(t) \equiv\min_j \rho_j(t)$ is smaller than $\rho$, and let $\CJ_*(t)
\equiv\{j\dvtx\rho_j(t) = \rho_*(t)\}$. Then all customer types
in $\CC(\CJ_*(t)) \equiv\bigcup_{j\in\CJ_*(t)} \CC(j)$ are
routed to server pools in $\CJ_*(t)$, so the total arrival rate
``into'' $\CJ_*(t)$ is no less than nominal; on the other hand, since
$\mu_{ij} = \mu_j$ and server occupancy is lower than nominal, the
total departure rate ``from'' $\CJ_*(t)$ is smaller than nominal. This
shows that if $\rho_* < \rho-\varepsilon< \rho$, then $\dot\rho_* >
\delta> 0$, where $\delta\geq c \varepsilon$ for some constant $c>0$
(depending on the system parameters). That is, if $\rho_*(t) < \rho$,
then $\dot\rho_*(t) \geq c(\rho-\rho_*(t))$, so $\rho_*(t)$ is
bounded below by a function converging exponentially fast to $\rho$.

Consider a fixed, sufficiently small $\varepsilon>0$;
we know that $\rho_*(t) \geq\rho-\varepsilon$ for all large times $t$.
If some customer class $i$ has a queue $q_i(t) > 0$, then all server
classes $j \in\CS(i)$ have $\rho_j(t) = 1$. It is now easy to see
that the system is serving customers faster than they arrive (because
$\rho<1$ and $\varepsilon$ is small).
This easily implies that all $q_i(t)=0$ after a finite time.

In the absence of queues, we can analyze $\rho^*(t) = \max_j \rho_j(t)$
similarly to the way we treated $\rho_*(t)$; namely, we show that
$\rho^*(t)$ is bounded above by a function
converging exponentially fast to $\rho$, which tells us that $\rho_j(t)
\to\rho$ for all $j$. Once all $\rho_j(t)$ are close enough to
$\rho$,
we can use the argument essentially identical to that in the proof of
Theorem~\ref{thmunderload,linearequation} to conclude that, after a
further finite time, we will have $\rho_j(t) = \rho_{j'}(t)$ for all
$j$, $j'$.
[The argument is even simpler, because, unlike in Theorem \ref
{thmunderload,linearequation},
where it was required that $(\psi_{ij}(t))$ were close to nominal,
here it suffices that $(\rho_j(t))$ are close to nominal,
because of the $\mu_{ij}=\mu_j$ assumption.]
With $\rho(t)=\rho_j(t), \forall j$, we then have for the total
amount of ``fluid'' in the system
\[
(d/dt)\sum_j \beta_j \rho(t) = \sum
_i \lambda_i - \sum
_j \beta_j \rho(t) \mu_j.
\]
This is a simple linear ODE for $\rho(t)$, which implies that (after a
finite time) $\rho(t)-\rho= c_1 \exp(-c_2 t)$, with constant $c_2>0$
and $c_1$.
This in particular means that $\dot\rho_j(t)=\dot\rho(t) \to0$.
Denote by $\hat\lambda_{ij}(t)$ the rate at which fluid $i$ arrives
at pool $j$, namely
%
%
\begin{equation}
\label{eq-psi-ij} \hat\lambda_{ij}(t)=\mu_j
\psi_{ij}(t) + \dot\psi_{ij}(t);
\end{equation}
at any large $t$ we have $\sum_j \hat\lambda_{ij}(t)=\lambda_i$.
Then, for each $j$,
\begin{eqnarray*}
\sum_i \hat\lambda_{ij}(t) &=& \sum
_i \mu_j \psi_{ij}(t) +
\sum_i \dot\psi_{ij}(t) =
\beta_j \mu_j \rho_j(t) +
\beta_j \dot\rho_j(t)
\\
&\to&\beta_j \mu_j \rho= \sum
_i \lambda_{ij}.
\end{eqnarray*}
This is only possible if each $\hat\lambda_{ij}(t)\to\lambda_{ij}$.
But then the ODE (\ref{eq-psi-ij}) implies $\psi_{ij}(t)\to\psi_{ij}^*$.

Now, consider a critically loaded system, $\rho= 1$. Essentially same argument
as above tells us that, as long as not all queues $q_i(t)$ are equal,
each of the longest queues
gets more service than the arrival rate into it,
and so $q^*(t)=\max q_i(t)$ has strictly negative, bounded away from $0$
derivative. If all $q_i(t)$ are equal and positive, then $\dot{q}^*(t)=0$.
We see that $q^*(t)$ is nonincreasing, and so $q^*(t)\downarrow q \ge0$.
We also have $\rho_*(t)\to\rho=1$ exponentially fast.
(Same proof as above applies.)
These facts easily imply convergence to an equilibrium point.
We omit further detail.

Examination of the above proof shows that it implies the following property,
for both cases $\rho<1$ and $\rho=1$. For any fixed equilibrium point
(with $q>0$ if $\rho=1$), there exists a sufficiently small
$\varepsilon
>0$ such that
for all sufficiently small
$\delta>0$, any fluid model starting in the $\delta$-neighborhood
of the equilibrium point, first, never leaves the $\varepsilon$-neighborhood
of the equilibrium point and, second, converges to an equilibrium point
(possibly different from the ``original'' one, if $\rho=1$).
This property cannot hold, unless the system is locally stable; see
Section~\ref{sec-local-stability-def}.
\end{pf}

\subsection{\texorpdfstring{Local stability in the case $\mu_{ij} = \mu_i$, $(ij)\in\CE$}
{Local stability in the case mu ij = mu i, (ij) in E}}
\label{seclocalstabilityformuij=mui}

%
\begin{theorem}
\label{th-loc-stab-spec2}
Assume $\rho<1$ and $\mu_{ij} = \mu_i$ for $(ij)\in\CE$.
Then the system is locally stable (i.e., $A_u$ is stable).
\end{theorem}
\begin{pf}
We have
\[
\dot{\psi}_i(t) = \lambda_i - \mu_i
\psi_i(t)
\]
and $A_u$ is simply a diagonal matrix with entries $-\mu_i$.
%
\end{pf}
%
%
\begin{theorem}
\label{th-loc-stab-spec222}
Assume $\rho= 1$ and $\mu_{ij} = \mu_i$ for $(ij) \in\CE$. Then
the system is locally stable (i.e., $A_c$ is stable).
\end{theorem}
\begin{pf}
As seen in the proof of Theorem~\ref{th-loc-stab-spec2}, the matrix
$A_u$ in this case is diagonal with entries $-\mu_i$. By Lemma \ref
{lemmaentriesofAc}, $A_c$ has off-diagonal entries $(A_c)_{ii'} =
\mu_{i'}/I$ and diagonal entries $-\mu_i(1-1/I)$. That is, its
off-diagonal entries are strictly positive. Therefore, $A_c + \eta I$
for some large enough constant $\eta>0$ (where $I$ is the identity
matrix) is a positive matrix. By the Perron--Frobenius theorem (\cite
{Meyer}, Chapter 8), $A_c + \eta I$ has a real eigenvalue $p + \eta$
with the property that any other eigenvalue of $A_c + \eta I$ is
smaller than $p+\eta$ in absolute value (and in particular has real
part smaller than $p + \eta$). Moreover, the associated \textit{left}
eigenvector $w$
is strictly positive, and is the unique (up to scaling) strictly
positive left eigenvector of $A_c + \eta I$.
Translating these statements to $A_c$, we get: $A_c$ has a real
eigenvalue~$p$; all other eigenvalues of $A_c$ have real part smaller
than $p$; $A_c$ has unique (up to scaling) strictly positive left
eigenvector $w$; and the eigenvalue of $w$ is $p$.

Now, $A_c$ has a positive left eigenvector with eigenvalue $0$,
namely $(1,1,\break\ldots,1)$. Therefore, we must have $p = 0$, and we
conclude that all other (i.e., nonzero)
eigenvalues of $A_c$ have real part smaller than $0$, as required.
\end{pf}

\section{\texorpdfstring{Fluid models for general $\mu_{ij}$: Local instability examples}
{Fluid models for general mu ij: Local instability examples}}
\label{secunderloaddivergence}

In Sec-\break tions~\ref{secstabilityformuij=muj},~\ref{seclocalstabilityformuij=mui}
we have shown that the matrices $A_u$ and $A_c$ are stable in the cases
$\mu_{ij} = \mu_j$, $(ij) \in\CE$ and $\mu_{ij} = \mu_i$, $(ij) \in\CE$.
Since the entries of $A_u$, $A_c$ depend continuously on $\mu_{ij}$
via Lemmas~\ref{lemmaentriesofAu},~\ref{lemmaentriesofAc} and
the eigenvalues of a matrix depend continuously on its entries, we know
that the matrices will be stable for all parameter settings
sufficiently close to those special cases. Therefore, there exists a
nontrivial parameter domain of local stability.
One might consider it to be a reasonable
conjecture that local stability holds for any parameters.
It turns out, however, that this conjecture is false.
We will now construct examples to demonstrate that, in general, the
system can be locally unstable.
%
%
\begin{remark}\label{remchoiceoflambda}
In the examples below, we will specify the parameters $\mu_\CE$ and
sometimes $\beta_\CJ$, but not $\lambda_\CI$. It is easy to
construct values of $\lambda_\CI$ which will make all of the
activities in $\CE$ basic; simply pick a strictly positive vector
$\psi_\CE$,
such that all loads $\sum_i \psi_{ij}/\beta_j$ are equal,
and set $\lambda_i = \sum_j \psi_{ij} \mu_{ij}$. Lemmas~\ref
{lemmaentriesofAu}(iii) and~\ref{lemmaentriesofAc}(iii)
guarantee that the specific values of $\lambda_\CI$ do not affect the
matrices $A_u$, $A_c$. In critical load, we also do not need to specify~$\beta_\CJ$.
\end{remark}

\textit{Local instability example} 1. Consider a system with 3 customer
types $A$, $B$, $C$ and 4 server types $1$ through $4$, connected \mbox{$1 -
A - 2 - B - 3 - C - 4$}. Set $\beta_1 = 0.97$ and $\beta_2 = \beta_3 =
\beta_4 = 0.01$. Set $\mu_{A1} = \mu_{B2} = \mu_{C3} = 1$ and
$\mu_{A2} = \mu_{B3} = \mu_{C4} = 100$. (See Figure \ref
{figunderloadcounterexample1}.)
%
%
\begin{figure}

\includegraphics{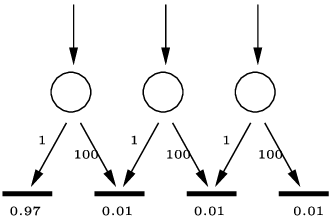}

\caption{System with three customer types whose underload equilibrium
is unstable.}
\label{figunderloadcounterexample1}
\end{figure}
%
On the other hand, we compute by Lem\-ma~\ref{lemmaentriesofAu}
\[
A_u = %
\pmatrix{ -1.99 & -0.99 & -0.99
\cr
97.02 & -2.98
& -1.98
\cr
96.03 & 96.03 & -3.97 } %
\]
with eigenvalues $\{-17.8, 4.45 \pm23.4i\}$.
Therefore by Theorem~\ref{thmunderload,linearequation}, the system
with these parameters is described by an unstable ODE in the
neighborhood of its equilibrium point.

We now show that this is a minimal instability example, in the sense
made precise
by the following:
%
%
\begin{lemma}\label{lemmaunderloadlocalstabilityfortwocustomertypes}
Consider an underloaded system, $\rho<1$.

\begin{longlist}
\item
Let $I\ge2$.
Any customer type $i$ that is a leaf in the basic activity tree, does
not affect the local stability of the system. Namely, let us modify the
system by removing type $i$,
and then modifying (if necessary) input rates $\lambda_{k}$ of
the remaining types $k\in\CI\setminus i$
so that the basic activity tree of the modified system is $\CE
\setminus(ij)$,
where $(ij)$ is the (only) edge in $\CE$ adjacent to $i$.
Then, the original system
is locally stable if and only if the modified one is.

\item A system with two (or one) nonleaf customer types is locally stable.
\end{longlist}
\end{lemma}
\begin{pf}
(i) If type $i$ is a leaf, the equation for $\psi_i(t)$ is simply
$\dot{\psi}_i(t)= \lambda_i - \mu_{ij} \psi_i(t)$.
This means (setting $i=1$) that $(1,0,\ldots,0)^{\dagger}$ is an
eigenvector of
$A_u$ with eigenvalue $-\mu_{ij}$. Further, it is easy to see that: (a)
the rest of the eigenvalues of $A_u$ are those of matrix $A_u^{(-i)}$ obtained
from $A_u$ by removing the first row and first column; and (b)
$A_u^{(-i)}$ is exactly
the ``$A_u$-matrix'' for the modified system.

(ii) We can assume that there are no customer-type leaves.
The case $I=1$ is trivial\vadjust{\goodbreak} (and is covered by Theorem \ref
{thmglobalstabilityformuij=muj}), so let $I=2$. Throughout the proof, the
pool sizes $\beta_j$ are fixed.
From Theorem~\ref{thmglobalstabilityformuij=muj}
we know that for a certain set of service rate values
[namely, $\mu_{ij} = \mu_j$, $(ij)\in\CE$],
the matrix $A_u$ is stable.
Suppose that we continuously vary the parameters $\mu_{ij}$ from those
initial values to
the values of interest, without ever making $\mu_{ij} = 0$.
If we assume that the final matrix $A_u$ is \textit{not} stable, then
as we change $\mu_{ij}$ the (changing) matrix $A_u$ acquires at some point
two purely imaginary eigenvalues.
If the eigenvalues of $A_u$ are purely imaginary, we must have $\trace
(A_u) = 0$.
However, as seen from the form of $A_u$ in Lemma~\ref{lemmaentriesofAu},
the diagonal entries of $A_u$ are always negative, and therefore
$\trace(A_u) < 0$.
The contradiction completes the proof.
\end{pf}

An argument similar to the above proof also allows us to explain how the
instability example 1 was found.
In degree 3, let the characteristic polynomial of $A_u$ be $x^3 - c_2
x^2 + c_1 x - c_0$.
\textit{A necessary and sufficient condition for all roots of the polynomial
to have negative real parts is}: \textit{$- c_2, c_1, - c_0 >0$ and
$c_2 c_1 <
c_0$}; \textit{see}~\cite{Farkas}, \textit{Theorem} 6. A necessary and sufficient
condition for
the ``boundary case'' between stability and instability (i.e., the
condition for a pair of conjugate purely imaginary roots) is $c_2 c_1 = c_0$.
Using Lemma~\ref{lemmaentriesofAu} we can evaluate the
characteristic polynomial symbolically and use the resulting expression
to find parameters for which
$c_2 c_1 = c_0$
will hold. See~\cite{onlinecomputations} online for the computations.

It is possible to construct an instability example
with more reasonable values of $\beta_j$, $\mu_{ij}$, although it
will be bigger. Figure~\ref{figunderloadcounterexample2} shows the
diagram. The associated matrix $A_u$ and its eigenvalues can also be
found online~\cite{onlinecomputations}.

%
%
\begin{figure}[b]

\includegraphics{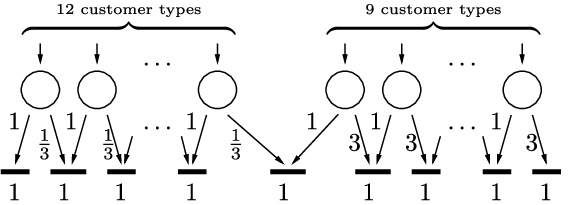}

\caption{System with $\beta_j = 1$ and $\mu_{ij} \in\{1/3,1,3\}$
whose underload equilibrium is unstable. There are 21 customer types.}
\label{figunderloadcounterexample2}
\end{figure}

We do not have an explicit characterization of the local instability
domain, beyond the necessity of $I\ge3$.

We now analyze the critically loaded system $\rho= 1$ with queues,
that is, the stability of the matrix $A_c$.
Recall that the transformation $A_c$, restricted to subspace $\{y
|\sum_i y_i=0\}$,
and then the stability of $A_c$,
does not depend on the values of $\beta_j$,
so it suffices to specify the values $\mu_{ij}$.

\textit{Local instability example} 2.
Consider the network of Figure~\ref{figcriticalloadcounterexample},
which has 5 customer types $A$ through $E$ and 4 server types $1$
through $4$, connected $A - 1 - B - 2 - C - 3 - D - 4 - E$, with the
following parameters:
\begin{eqnarray*}
\mu_{A1} &=& 1,\qquad  \mu_{B1} = 100,\qquad \mu_{B2} = 1,\qquad \mu_{C2} =
100,\\
\mu_{C3} &=& 1,\qquad \mu_{D3} = 100,\qquad \mu_{D4} = 10\mbox{,}000,\qquad \mu_{E4}
= 100.
\end{eqnarray*}

%
\begin{figure}

\includegraphics{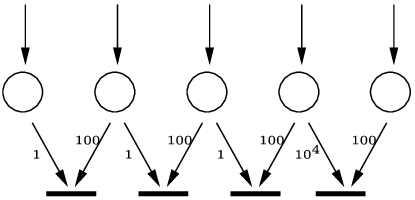}

\caption{System with five customer types whose critical load
equilibrium is unstable.}
\label{figcriticalloadcounterexample}
\end{figure}
%
The matrix $A_c$, computed from Lemma~\ref{lemmaentriesofAc} will
be given by
\[
A_c = \frac{1}{20} %
\pmatrix{ 9389 & 9805 & 10\mbox{,}201 &
10\mbox{,}597 & -29\mbox{,}003
\cr
10\mbox{,}894 & 9290 & 9706 & 10\mbox{,}102 & -29\mbox{,}498
\cr
10\mbox{,}399 & 10\mbox{,}795 & 9191 &
9607 & -29\mbox{,}993
\cr
-40\mbox{,}091 & -39\mbox{,}695 & -39\mbox{,}299 & -40\mbox{,}903 & 119\mbox{,}497
\cr
9409 & 9805 &
10\mbox{,}201 & 10\mbox{,}597 & -31\mbox{,}003 } %
\]
and the eigenvalues of $A_c$ are $\{0, -16.88, -2190.05, 2.565 \pm
23.23i\}$.

Again, the above example 2 is in a sense minimal:
%
%
\begin{lemma}
Consider a critically loaded system, $\rho=1$.

\begin{longlist}
\item
Let $J\ge2$.
Any server type $j$ that is a leaf in the basic activity tree does not
affect the local stability of the system. Namely, let us modify the
system by removing type $j$,
and then replacing $\lambda_i$ for the unique $i$ adjacent to $j$ by
$\lambda_i -
\beta_j \mu_{ij}$. Then, the original system
is locally stable if and only if the modified one is.

\item Consider a system labeled $S$. We say that a system $S'$ is an
\textup{expansion}
of system $S$ if it is obtained from $S$ by the following modification.
We pick one server type $j$ and one customer type $i$ adjacent to it in
$\CE$;
we ``split'' type $j$ into two types $j'$ and $j''$;
we ``connect'' type $i$ to both $j'$ and $j''$; each of the remaining types
$i' \in\CC(j) \setminus i$ we connect to either $j'$ or $j''$ (but
not both);
if $(i'j')$ [resp., $(i'j'')$] is a new edge, we set $\mu_{i'j'}=\mu_{i'j}$
(resp., $\mu_{i'j''}=\mu_{i'j}$). Then, $S$ is locally stable if and
only if $S'$ is.

\item A system with four or fewer customer types is locally stable.
\end{longlist}
\end{lemma}
\begin{pf}
(i) The argument here is a ``special case'' of the one used to show
the independence of transformation
$A_c$ [restricted to $(I-1)$-dimensional invariant subspace]
from $(\beta_j)$ in the proof of Lemma~\ref{lemmaentriesofAc}.
Namely, it is easy to check that
the original system and the modified system share exactly same
ODE (\ref{eq-ac3}).

\mbox{}\hphantom{i}(ii) Again, it is easy to see that the two systems share the same ODE~(\ref{eq-ac3}).

(iii) We can assume that there are no server-type leaves, so that the
tree $\CE$ has only customer-type leaves, of which it can have two,
three, or four.

If it has four customer-type leaves, then the tree has a total of four
edges, hence five nodes, that is, a single server pool, to which all
the customer types are connected.

%
%
\begin{figure}[b]

\includegraphics{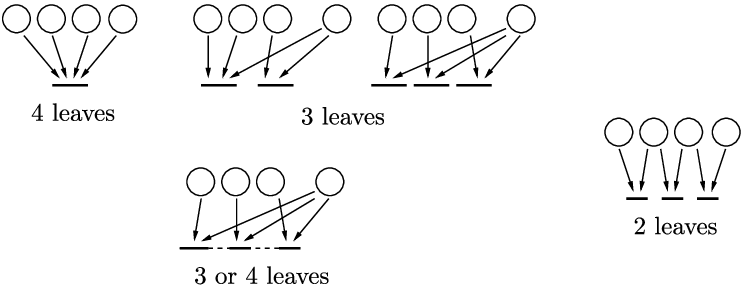}

\caption{Possible arrangements of four customer types.}
\label{figshapesoftreeswith4customertypes}
\end{figure}

If the tree has three customer-type leaves, then letting $k$ be the
number of edges from the fourth customer type, we have $k+3$ total
edges, so $k+4$ nodes, of which $k$ are server types. That is, the
nonleaf customer type is connected to all of the server types. Since
there are no server-type leaves, we must have $k \leq3$; since we are
assuming the fourth customer type is not a leaf, we must have $k \geq
2$; thus, $k = 2$ or $k=3$.

The last case is of two customer-type leaves. Letting $k, l$ be the
number of edges coming out of the other customer types, we have $k + l
+ 2$ edges. On the other hand, since each server type has at least 2
edges coming out of it, we have at most $(k+l+2)/2$ server types, so at
most $(k+l+2)/2+4$ nodes. Thus, we have $(k+l+2)+1 \leq(k+l+2)/2 + 4$,
or $k+l+2 \leq6$, giving $k=l=2$ (since they must both be $\geq2$).

We summarize the possibilities in Figure \ref
{figshapesoftreeswith4customertypes}.
Note that the bottom-left system can be obtained by a sequence of
expansions from each of the top-left systems, and so this is the only system
we need to consider to establish local stability for all 3- and 4-leaf cases.
Thus, in total, the only two systems that need to be considered are
bottom-left and right.
In both of the resulting cases, we can use Lemma~\ref{lemmaentriesofAc}\vadjust{\goodbreak}
to write out $A_c$ and its characteristic polynomial explicitly.
The characteristic polynomial will have degree 4, but one of its roots
is 0, so we can reduce it to degree 3.
We then symbolically verify that the cited above stability criterion
(\cite{Farkas}, Theorem 6) for degree 3 polynomials, is satisfied. See
\cite{onlinecomputations} online for the details.
\end{pf}

An argument similar to that in the above proof
allows us to explain how the
instability example 2 was found.
We seek a condition satisfied by the coefficients of a degree 4 polynomial
with two imaginary roots. Letting the polynomial be $x^4 - c_1 x^3 +
c_2 x^2 - c_3 x + c_4$, and letting the roots be $\eta_1$, $\eta_2$,
$\pm iz$ (where $\eta_1$ and $\eta_2$ may be real or complex
conjugates, and $z \in\BR$), we see that $c_1 = \eta_1 + \eta_2$,
$c_2 = \eta_1 \eta_2 + z^2$, $c_3 = (\eta_1 + \eta_2) z^2$ and $c_4
= \eta_1 \eta_2 z^2$.
This implies the relation
$c_4 c_1^2 + c_3^2 - c_1 c_2 c_3 = 0$, and we can find the parameters
for which this is true. (The symbolic calculation will involve rather a
lot of terms.)
We remark that, whereas for degree 3 polynomials the condition $c_2 c_1
- c_0=0$
is necessary and sufficient for the existence of two imaginary roots
(\cite{Farkas}, Theorem 6),
the condition we derive here for degree 4 polynomials is necessary, but
not sufficient.
[E.g., the polynomial $(x-1)^2(x+1)^2$ has $c_1 = c_3 = 0$, so
$c_4 c_1^2 + c_3^2 - c_1 c_2 c_3 = 0$, but it has no imaginary roots.]
Thus, checking the sign of
the corresponding expression alone is insufficient to determine whether
the system is unstable, but is a useful way of narrowing down the
parameter ranges.

Finally, it is possible to construct a single system which will be
unstable both for $\rho< 1$ and for $\rho= 1$ with positive queues.
For the local stability of the underloaded system, the leaves of the
basic activity tree corresponding to customer types are irrelevant (the
corresponding occupancy on the sole available server class converges to
nominal exponentially). On the other hand, for the critically loaded
system, the leaves corresponding to server pools are irrelevant, since
the corresponding server is fully occupied by its unique available
customer type. This observation allows us to merge the above two
systems into a single one which is unstable both in underloaded and in
the critically loaded case.

Consider a system with 5 customer types $A$ through $E$ and 5 server
types $0$ through $4$ connected as $0 - A - 1 - B - 2 - C - 3 - D - 4 -
E$, with $\mu_{A0}=100$ and the remaining $\mu_{ij}$ as in the
critically loaded case. Set $\beta_3 = 0.96$ while $\beta_0, \beta_1,
\beta_2, \beta_4 = 0.01$; see Figure \ref
{figunderloadandcriticalloadcounterexample}.
%
%
\begin{figure}

\includegraphics{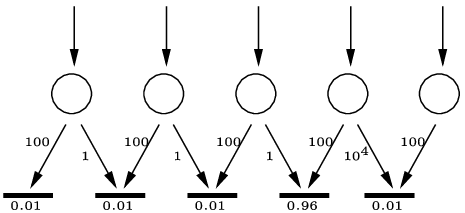}

\caption{System with five customer types whose underload and critical
load equilibrium points are both unstable.}
\label{figunderloadandcriticalloadcounterexample}
\end{figure}
By the above discussion, this system must be unstable for $\rho= 1$
and positive queues. We therefore need to consider only the first 4
customer types ($E$ is a customer-type leaf and does not matter) in
underload. We compute
\[
A_{u} = %
\pmatrix{ -1.99 & -0.99 & -0.99 & -0.99
\cr
97.02
& -2.98 & -1.98 & -1.98
\cr
96.03 & 96.03 & -3.97 & -2.97
\cr
-99 & -99 & -99 &
-199 } %
\]
and eigenvalues are $\{-14.6, -201.1, 3.91 \pm18.1 i\}$.\vadjust{\goodbreak}

While we showed above that sufficiently small systems are at least
locally stable,
we will show now that, in the underload case,
any sufficiently large system is locally unstable for some parameter settings.
%
%
\begin{lemma}
In underload ($\rho< 1$), any shape of basic activity tree that
includes a
locally unstable system (i.e., with $A_u$ having an eigenvalue with
positive real part)
as a subset will, with some set of parameters $(\beta_j)$, $(\mu
_{ij})$, become locally unstable. In particular, any shape of basic
activity tree that includes
instability
example 1 (Figure~\ref{figunderloadcounterexample1}) above (for
$\rho<1$) will be locally
unstable for some set of parameters $\beta_j$, $\mu_{ij}$.
\end{lemma}
\begin{pf}
Let $U$ be any system whose underload ($\rho< 1$) equilibrium is
locally unstable,
for example, one of the examples given above, with the
associated fixed set of parameters $\mu_{ij}$, $\beta_j$ and $\lambda_i$.
Let $S$ be a system including $U$ as a subset, namely:
the activity tree of $S$ is a superset of that of $U$;
the $\mu_{ij}$ and $\beta_j$ in $U$ are preserved in $S$;
the $\mu_{ij}$ in $S$ are fixed.
Consider a sequence of systems $S^\varepsilon$ in which $\beta_j =
\varepsilon\to0$ for all $j$ not in $U$.
For each $\varepsilon$, take $\lambda_{i}^{\varepsilon}$
so that all of the activities are indeed basic,
and such that, as $\varepsilon\to0$, $\lambda_{i}^{\varepsilon}\to
\lambda_i$
for $i$ in $U$, and $\lambda_{i}^{\varepsilon}\to0$ for $i$ not in $U$;
see Remark~\ref{remchoiceoflambda}.
Order the $\psi_i$ so that the customer types $i$ in $U$ come first.
Suppose there are $I$ customer types in $U$, and $I+k$ customer types
in $S$. Let $A_u^\varepsilon$ be the $(I+k) \times(I+k)$ matrix
associated with $S^\varepsilon$, and let $A_u$ be the $I \times I$ matrix
associated with $U$ considered as an isolated
system.
Then as $\varepsilon\to0$ the top left $I \times I$ entries of
$A_u^\varepsilon$ converge to $A_u$, while the bottom left $k \times I$
entries of $A_u^\varepsilon$ converge to 0
(i.e., the effect of $U$ on the stability of
the rest of the system vanishes---this is due to the fact that pool
size parameters
$\beta_j$ in $U$ remain constant, while $\beta_j\to0$ in the rest of
the system).
Consequently, each eigenvalue of $A_u$ is a limit of eigenvalues of
$A_u^\varepsilon$. Since $A_u$ had an eigenvalue with positive real part,
for sufficiently small $\varepsilon$ the matrix $A_u^\varepsilon$
will have
at least one eigenvalue with positive real part as well, so the system
$S^\varepsilon$ will be locally unstable.
\end{pf}

\section{Diffusion scaled process in an underloaded system. Possible
evanescence of invariant distributions}\label{sectiondiffusionlimit}

Above we have shown that on a fluid scale, around the equilibrium
point, the system converges to a subset of its possible states, on
which it evolves according to a differential equation, possibly
unstable. This strongly suggests that, when the differential
equation is unstable, the stochastic system is in fact ``never'' close
to equilibrium.
Our goal in this section is to demonstrate that it is the case at least
on the diffusion scale. More precisely,
we consider the system in underload, $\rho< 1$, and look at
diffusion-scaled stationary distributions (centered at the equilibrium
point and scaled down
by $\sqrt{r}$); we show that, when the associated fluid model is
locally unstable, this sequence of stationary distributions is such
that the measure of any compact set vanishes.

\subsection{Transient behavior of diffusion scaled process.
State space collapse}
\label{sec-diffusion-finite}

In this section we cite the diffusion limit result
(for the process transient behavior)
that we will need from~\cite{GurvichWhitt}.
Again, we consider a sequence of systems indexed by~$r$, with
the input rates being $\lambda^r_i = r \lambda_i$, server pool sizes being
$\beta_j r$, and the service rates $\mu_{ij}$ unchanged with $r$.
[Here we drop the $o(r)$ terms in $\lambda^r_i = r \lambda_i + o(r)$,
because, when $\rho<1$, considering these terms does not make sense.]
The notation for the unscaled processes is the same as in the previous
section; however, we are now interested in a different---diffusion---scaling.
We define
%
%
\begin{eqnarray}
\label{eqndiffusionscale} \hat\Psi^r_{ij}(t) &=&
\frac{\Psi^r_{ij}(t) - r \psi_{ij}^*}{\sqrt{r}},\qquad \hat\Psi^r_i(t)
= \sum
_j \hat\Psi^r_{ij}(t),
\nonumber\\[-8pt]\\[-8pt]
\hat\Psi^r_j(t) &=& \sum_i
\hat\Psi^r_{ij}(t) = \frac{\Psi^r_j(t)
- \rho r \beta_j}{\sqrt{r}}.
\nonumber
\end{eqnarray}
We will denote by $M'$ the linear mapping from $z= (z_{ij}, (ij)\in\CE
) \in\BR^{I+J-1}$ to $y=(y_i)\in\BR^I$, given by $\sum_j z_{ij} =
y_i$. [So,
$(\hat\Psi^r_i(t)) \equiv M' (\hat\Psi^r_{ij}(t))$.] There is the
obvious relation
between $M'$ and the operator $M$ defined by (\ref{M-def}): $M'M y = y$
for any
$y\in\BR^I$. Let us define $\CM:= \{My | y\in\BR^I\}$, an
$I$-dimensional linear subspace of $\BR^{I+J-1}$; equivalently,
$\CM= \{z\in\BR^{I+J-1} | z=MM'z\}$.
%
%
\begin{theorem}[(Essentially a corollary of Theorems 3.1 and 4.4
in~\cite{GurvichWhitt})]\label{thmdiffusionSDE}
Let $\rho< 1$. Assume that as $r\to\infty$, $\hat\Psi^r_\CE(0)\to
\hat\Psi_\CE(0)$ where $\hat\Psi_\CE(0)$ is deterministic and
finite. [Consequently, $\hat\Psi^r_\CI(0)\to\hat\Psi_\CI(0)= M'
\hat\Psi_\CE(0)$.] Then,
%
%
\begin{equation}
\label{eq-psi-i-conv} \hat\Psi^r_\CI(\cdot) \implies\hat
\Psi_\CI(\cdot) \qquad\mbox{in $D^{I}[0,\infty)$}
\end{equation}
and for any fixed $\eta>0$,
%
%
\begin{equation}
\label{eq-psi-ij-conv} \hat\Psi^r_\CE(\cdot) \implies M
\hat\Psi_\CI(\cdot) \qquad\mbox{in $D^{I+J-1}[\eta,\infty)$},
\end{equation}
where
$\hat\Psi_\CI(\cdot)$ is the unique solution of the SDE
%
%
\begin{equation}
\label{eq-psi-limit}\qquad \hat\Psi_i(t) = \hat\Psi_i(0) -
\sum_{j \in\CS(i)} \mu_{ij} \int
_0^t \bigl(M \hat\Psi_\CI(s)
\bigr)_{ij} \,ds + \sqrt{2\lambda_i} B_i(t),\qquad
i\in\CI,
\end{equation}
and the processes $B_i(\cdot)$ are independent standard Brownian motions.
\end{theorem}

Recalling the definition of matrix $A_u$ [see (\ref
{eqnunderloadODEmatrix})], (\ref{eq-psi-limit}) can be written as
%
%
\begin{equation}
\label{eqnformofODE} \hat\Psi_\CI(t) = \hat\Psi_\CI(0)
+ \int_0^t A_u \hat
\Psi_\CI(s) \,ds + \bigl(\sqrt{2\lambda_i}
B_i(t)\bigr).
\end{equation}

The meaning of Theorem~\ref{thmdiffusionSDE} is simple: the
diffusion limit of the process $\hat\Psi^r_\CI(\cdot)$ is such
that, at initial time $0$, it ``instantly jumps'' to the state $MM'
\hat\Psi_\CE(0)$ on the manifold $\CM$ [where $MM' \hat\Psi_\CE
(0) = \hat\Psi_\CE(0)$ only if $\hat\Psi_\CE(0)\in\CM$]; after
this initial jump, the process stays on $\CM$ and evolves according to
SDE (\ref{eqnformofODE}). Theorem~\ref{thmdiffusionSDE} is
``essentially a corollary'' of results in~\cite{GurvichWhitt},
because the setting in~\cite{GurvichWhitt} is such that $\rho=1$,
while we assumed $\rho<1$. However, our Theorem~\ref{thmdiffusionSDE}
can be proved the same way, and in a sense is easier, because when
$\rho<1$, the queues vanish in the limit (which is why the queue
length process is not even present in the statement of Theorem
\ref{thmdiffusionSDE}).

\subsection{Evanescence of invariant measures}\label
{sectioninvariantmeasures}

In this section we show that if the matrix $A_u$ has eigenvalues with
positive real part, the stationary distribution of the (diffusion
scaled) process $\hat\Psi^r_\CE(\cdot)$ escapes to infinity as
$r\to\infty$. Namely, we prove the following:
%
%
\begin{theorem}
\label{th-measure-escapes}
Suppose\vspace*{1pt} $\rho<1$.
Consider a sequence of systems as defined in Section
\ref{sec-diffusion-finite},
and denote by $\mu^r$ the stationary distribution of the
process $\hat\Psi^r_\CE(\cdot)$, a probability measure on $\BR^{I+J-1}$.
Let $b_K = \{|z| \leq K\} \subset\BR^{I+J-1}$.
Suppose the matrix $A_u$ has eigenvalues with positive real parts and
no pure imaginary eigenvalues.\footnote{The requirement of ``no pure
imaginary eigenvalues'' is made for convenience of differentiating
between strict convergence and strict divergence. It holds for generic
values of $\beta_j$, $\mu_{ij}$: that is, any set of values $\beta_j$,
$\mu_{ij}$ has a small perturbation $\tilde{\beta}_j$, $\tilde
{\mu}_{ij}$ with for which $A_u$ has no pure imaginary eigenvalues.}
Then for any $K$, $\mu^r(b_K) \to0$ as $r\to\infty$.
\end{theorem}

Before we proceed with the proof, let us introduce more notation and
one auxiliary result.
Let $\CC_\CI$ be the submanifold of convergence (stability)
of ODE $(d/dt)y=A_u y$ on $\BR^{I}$; namely, $\CC_\CI$ is the (real)
subspace of $\BR^{I}$
spanned by the Jordan basis vectors
for matrix $A_u$ corresponding to all eigenvalues with negative real
parts. Given assumptions of
the theorem on $A_u$, the solutions to $(d/dt)y = A_u y$ converge to
$0$ exponentially fast if $y(0)\in\CC_\CI$, and go to infinity
exponentially fast
if $y(0)\in\BR^{I}\setminus\CC_\CI$.
Let $\CC=M \CC_\CI$ denote the corresponding submanifold of
convergence (stability)
of the linear ODE $(d/dt)z = (MA_uM') z$ on $z\in\CM$. This ODE is
just the $M$-image
of ODE $(d/dt)y=A_u y$. Therefore, a solution $z(t)$ converges to $0$
exponentially fast if $z(0)\in\CC$, and goes to infinity
exponentially fast
if $z(0)\in\CM\setminus\CC$. Let us denote
$b_K(\delta_1,\delta_2):= b_K \cap\{d(z,\CM)\le\delta_1, d(z,\CC
)\ge\delta_2\}$,
where $d(\cdot,\cdot)$ is Euclidean distance.
%
%
\begin{lemma}\label{lem-sde}
Solutions to SDE (\ref{eqnformofODE}) have the following
properties:

\begin{longlist}
\item
For any $T>0$ and any $\Psi_\CI(0)$,
\[
\BP\bigl\{M \hat\Psi_\CI(T) \in\CM\setminus\CC\bigr\}=1;
\]

\item For any $K>0$, $\delta_2>0$ and $\varepsilon>0$,
there exist sufficiently large $T_K$ and $K'>K$, such that,
uniformly on $M \hat\Psi_\CI(0) \in b_K(0,\delta_2)$,
\[
\BP\bigl\{M \hat\Psi_\CI(T_K) \in b_{K'}
\setminus b_{2K}\bigr\} \ge1-\varepsilon.
\]
\end{longlist}
\end{lemma}
\begin{pf}
Statement (i) follows from the fact that, regardless of the (deterministic)
initial state
$\Psi_\CI(0)$, the solution to SDE (\ref{eqnformofODE}) is such that
the distribution of $\Psi_\CI(T)$ is Gaussian with nonsingular covariance
matrix. (See~\cite{KaratzasShreve}, Section 5.6. In our case the matrix
of diffusion coefficients is diagonal with entries $\sqrt{2\lambda_i}$.)
Therefore, the probability that $\Psi_\CI(T)$ is in a subspace of
lower dimension
is zero.

Statement (ii) follows from the fact (again, see
\cite{KaratzasShreve}, Section
5.6) that the expectation $m(t) = \BE\hat\Psi_\CI
(t)$ evolves according to ODE
\[
\dot m(t) = A_u m(t).
\]
Since $d(M \hat\Psi_\CI(0), \CC)\ge\delta_2$ [and thus $\hat\Psi
_\CI
(0)$ is also separated by a positive distance from $\CC_\CI$],
we have
\[
\bigl\llvert m(t)\bigr\rrvert\geq a_1 \exp(at) 
\]
for some fixed $a_1,a>0$ and all large $t$. [Here $a_1$ depends on the
minimum length of the projection of $\hat\Psi_\CI(0)$ along $\CC
_\CI$
onto the (real) span of the Jordan basis vectors of $A_u$ corresponding
to eigenvalues with positive real part, and $a$ is the smallest
positive real part of an eigenvalue of $A_u$.] It is easy to check that
if the mean of a Gaussian distribution goes to infinity, then
(regardless of how the covariance matrix changes) the measure of any
bounded set goes to zero. On the other hand, both $m(t)$ and the
covariance matrix remain bounded for all $t\in[0,T_K]$, with any
$T_K$; then, for any $T_K$, we can always choose $K'$ large enough so
that $\BP\{M \hat\Psi_\CI(T_K) \in b_{K'}\}$ is arbitrarily close to
$1$.
\end{pf}
%
%
\begin{pf*}{Proof of Theorem~\ref{th-measure-escapes}}
We will consider measures $\mu^r$ as measures on the one-point
compactification $\overline{\mathbb{R}}{}^n=\mathbb{R}^n \cup\{*\}$
of the space $\mathbb{R}^n$, where $n=I+J-1$.
In this space, any subsequence of $\{\mu^r\}$ has a further subsequence,
along which $\mu^r \weakto\mu$ for some probability measure
$\mu$ on $\overline{\mathbb{R}}{}^n$. We will show that the entire measure
$\mu$ is concentrated on the infinity point $*$, that is, $\mu
(\mathbb{R}^n)=0$.
Suppose not, that is, $\mu(\mathbb{R}^n)>0$. The proof proceeds in
two steps.

\textit{Step} 1. We prove that $\mu(\mathbb{R}^n)=\mu(\CM\setminus
\CC)$.
Indeed, let us choose any $\varepsilon>0$, and $K$ large enough so that
$\mu(b_{K/2}) > (1-\varepsilon)\mu(\mathbb{R}^n)$. Then, for all
large $r$,
$\mu^r(b_{K}) > (1-\varepsilon)\mu(\mathbb{R}^n)$.
Choose $\delta_1>0$ and $T>0$ arbitrary.
From the properties of the limiting diffusion process (Lemma~\ref{lem-sde}),
we see that we can choose
a sufficiently small $\delta_2>0$ and sufficiently
large $K'$ such that, uniformly on the initial states
$\hat\Psi^r_\CE(0)\in b_K$,
\[
\liminf_{r\to\infty} \BP\bigl\{\hat\Psi^r_\CE(T) \in
b_{K'}(\delta_1,\delta_2)\bigr\} > 1-
\varepsilon.
\]
This implies that for all large $r$,
\[
\mu^r\bigl(b_{K'}(\delta_1,
\delta_2)\bigr) > (1-\varepsilon)^2 \mu\bigl(
\mathbb{R}^n\bigr),
\]
and then $\mu(b_{K'}(\delta_1,\delta_2)) \ge(1-\varepsilon)^2 \mu
(\mathbb{R}^n)$.
Since $\varepsilon$ and $\delta_1$ were arbitrary,
we conclude that $\mu(\mathbb{R}^n)\le\mu(\CM\setminus\CC)$,
and then, obviously, the equality must hold.

\textit{Step} 2. We show that, for any $K>0$,
$\mu(\mathbb{R}^n\setminus b_K)=\mu(\mathbb{R}^n)$. [This is, of
course, impossible
when $\mu(\mathbb{R}^n)>0$, and thus we obtain a contradiction.]
It suffices to show that for any $\varepsilon>0$, we can choose a sufficiently
large $K$, such that $\mu(\mathbb{R}^n\setminus b_K)\ge
(1-\varepsilon)^2 \mu(\mathbb{R}^n)$.
Let us choose (using step 1) a large $K$ and a small $\delta_2>0$,
such that $\mu(b_{K/2}(\delta_1/2,2\delta_2))> (1-\varepsilon)\mu
(\mathbb{R}^n)$
for any $\delta_1>0$.
Then, for any fixed $\delta_1>0$, for all large $r$,
$\mu^r(b_{K}(\delta_1,\delta_2))> (1-\varepsilon)\mu(\mathbb{R}^n)$.
Now, using Lemma~\ref{lem-sde}(ii), we can choose $K'$ and $T_K$ sufficiently
large, and then $\delta_1$ sufficiently small, so that, uniformly on
the initial states
$\hat\Psi^r_\CE(0)\in b_{K}(\delta_1,\delta_2)$,
\[
\liminf_{r\to\infty} \BP\bigl\{\hat\Psi^r_\CE(T_K)
\in b_{K'}\setminus b_{2K}\bigr\}\ge1-\varepsilon.
\]
Therefore,
\[
\mu^r(b_{K'}\setminus b_{2K}) > (1-
\varepsilon)^2 \mu\bigl(\mathbb{R}^n\bigr)
\]
for all large $r$, and then for the limiting measure $\mu$
we must have $\mu(\mathbb{R}^n\setminus b_K)\ge
(1-\varepsilon)^2 \mu(\mathbb{R}^n)$.
\end{pf*}

\section{Diffusion scaled process in a critically loaded system in
Halfin--Whitt asymptotic regime}
\label{sec-hw}

In this section we consider the following asymptotic regime.
The system is critically loaded, that is,
the optimal solution to SPP (\ref{eqnStaticLP})
is such that $\rho= 1$.
As scaling parameter
$r\to\infty$, assume that the server pool sizes are $r \beta_j$
(same as throughout the paper), and the input rates are
$\lambda^r_i = r \lambda_i + \sqrt{r} l_i$, where the parameters
(finite real numbers) $\{l_i\}$ are such that
$\sum l_i \nu_i = -C < 0$. Denote by
$\rho^r, \{\lambda^r_{ij}\}$ the optimal solution
of SPP (\ref{eqnStaticLP}), with $\beta_j$'s and $\lambda_i$'s
replaced by $r\beta_j$ and $\lambda_i^r$, respectively.
(This solution is unique, as can be easily seen from the CRP condition.)
Then, it is easy to check that $\rho^r = 1 + (\sum l_i \nu_i) / \sqrt{r}
= 1 - C/ \sqrt{r}$, which in turn easily implies that, for any $r$,
the system process is stable with the unique stationary
distribution.\looseness=1

We use the definitions of (\ref{eqndiffusionscale}) for the
diffusion scaled variables,
and add to them the following ones: $\hat X_i^r(t) = (X_i^r(t)-\psi_i^*
r)/\sqrt{r}$
for the (diffusion-scaled) number of type $i$ customers;
$\hat Q^r_i(t) = Q^r_i(t) / \sqrt{r}$ for the type $i$ queue length;
$\hat Z^r_j(t) = Z^r_j(t) / \sqrt{r}$, where $Z^r_j(t) = \Psi^r_j(t)
- r \beta_j \le0$
is the number\vspace*{2pt} of idle servers of type $j$ (with the minus sign).
Note that, although the optimal average
occupancy of pool $j$ is at $\rho^r r \beta_j$, the quantity $\hat
Z^r_j(t)$ measures the deviation from full occupancy $r \beta_j$.
Our choice of signs is such that $\hat Q^r_i \geq0$ while $\hat Z^r_j
\leq0$.
We will use the vector notations, such as $\hat X_{\CI}^r(t)$, as usual.

Two main results of this section are as follows:
(a) it is possible for the invariant distributions to escape to
infinity under certain
system parameters and (b) in the special case when service rate depends
on the server type only,
the invariant distributions are tight.

\subsection{Example of evanescence of invariant measures}
\label{sec-evanescence-in-HW}

Recall that $\pi$ denotes the (matrix of) orthogonal\vspace*{1pt} projection on the subspace
$L = \{y\in\BR^I |\break\sum_i y_i = 0\}$ in $\BR^I$; this is the
projection ``along'' the direction of vector $(1,\ldots,1)^{\dagger}$.
Also recall the relation between matrices $A_u$ and $A_c$,
\[
A_c=\pi A_u.
\]
One more notation: for $y\in\BR^I$,
\[
F[y]= \cases{ \pi y, &\quad if $\displaystyle\sum_i
y_i > 0$,
\vspace*{2pt}\cr
y, &\quad if $\displaystyle\sum_i
y_i\le0$.}
\]

Analogously to Theorem~\ref{thmdiffusionSDE},
the following fact is a corollary (this time---direct) of
Theorems 3.1 and 4.4 in~\cite{GurvichWhitt}.
%
%
\begin{theorem}
\label{thmdiffusionSDE-hw}
Assume that as $r\to\infty$, $\hat X^r_\CI(0)\to\hat X_\CI(0)$ and
$\hat\Psi^r_\CE(0)\to\hat\Psi_\CE(0)$,
where $\hat X_\CI(0)$ and $\hat\Psi_\CE(0)$ are deterministic and
finite. Then,
%
%
\begin{equation}
\label{eq-psi-i-conv-hw} \hat X^r_\CI(\cdot) \implies\hat
X_\CI(\cdot) \qquad\mbox{in $D^{I}[0,\infty)$}
\end{equation}
and for any fixed $\eta>0$,
%
%
\begin{equation}
\label{eq-psi-ij-conv-hw} \hat\Psi^r_\CE(\cdot) \implies
M F\bigl[\hat X_\CI(\cdot)\bigr] \qquad\mbox{in $D^{I+J-1}[\eta,
\infty)$},
\end{equation}
where
$\hat X_\CI(\cdot)$ is the unique solution of the SDE
%
%
\begin{equation}
\label{eq-psi-limit-hw} \hat X_\CI(t) = \hat X_\CI(0) +
\int_0^t A_u F\bigl[\hat
X_\CI(s)\bigr] \,ds + \bigl(\sqrt{2\lambda_i}
B_i(t)\bigr),
\end{equation}
and the processes $B_i(\cdot)$ are independent standard Brownian motions.
\end{theorem}

Next we establish the following fact.
%
%
\begin{lemma}
\label{lem-777}
There exists a system and a parameter setting such that the following
hold.

\begin{longlist}
\item
Matrix $A_c$ is unstable;

\item Matrix $A_u$ has $(1,\ldots,1)^{\dagger}$ as a right
eigenvector, with
real nonzero eigenvalue $c$,
%
%
\begin{equation}
\label{eq-key777} A_u (1,\ldots,1)^{\dagger} = c (1,\ldots
,1)^{\dagger}.
\end{equation}
\end{longlist}
\end{lemma}
\begin{pf}
Let us start with the system in the
local instability example~2
(see Figure~\ref{figcriticalloadcounterexample}) for the critical load.
We will modify it as follows. We will change $\mu_{D3}$ from $100$ to
$100-\varepsilon$
with sufficiently small positive $\varepsilon$, so that $A_c$ remains unstable.
(The reason for this change will be explained shortly.)
We will add two new server pools, 0 and 5, on the left and on the
right, respectively,
and set $\mu_{A0}=100$, \mbox{$\mu_{E5}=1$}; such addition of server-leaves
does not change
the instability of $A_c$. So, (i) holds.

Now, suppose all $\lambda_i$ are equal,
say $\lambda_i=1$. We can choose $\psi_{ij}^*$ such that
all $\psi_{i}^*= \sum_j \psi_{ij}^*$ are equal, and $\sum_j \mu_{ij}
\psi_{ij}^*=\lambda_i = 1$
for all $i$. Namely, we do the following. The reason for changing
$\mu_{D3}$ from $100$ to $100-\varepsilon$ is to make it possible to choose
$\psi_{D3}^*>0$ and $\psi_{D4}^*>0$, such that $\sum_j \mu_{Dj}
\psi_{Dj}^* = 1$
and $\psi_{D}^*= \psi_{D3}^*+ \psi_{D4}^* > 1/100$. We choose $\psi
_{A0}^*=1/100-\delta$,
$\psi_{A1}^*=100 \delta$ (which guarantees $\sum_j \mu_{Aj} \psi_{Aj}^*
= 1$) with $\delta>0$ small enough so that $\psi_{A}^* =
1/100 + 99\delta< 1/(100-\varepsilon)$.
The values of pairs $(\psi_{B1}^*,\psi_{B2}^*)$, $(\psi_{C2}^*,\psi
_{C3}^*)$,
$(\psi_{E4}^*,\psi_{E5}^*)$, are chosen to be equal to $(\psi
_{A0}^*,\psi_{A1}^*)$.
Finally, we choose $\psi_{D3}^*=(1-\delta_1)/(100-\varepsilon)$ and
$\psi_{D4}^*= \delta_1/10^4$ (which ensures $\sum_j \mu_{Dj} \psi
_{Dj}^* = 1$) with
$\delta_1>0$ satisfying
\[
\psi_{D}^*=(1-\delta_1)/(100-\varepsilon)+
\delta_1/10^4 = 1/100 + 99\delta=\psi_{A}^*.
\]
This completes the choice of $\psi_{ij}^*$.

We set $\beta_j = \sum_i \psi_{ij}^*$. We see that $(\psi_{ij}^*)$
is the equilibrium point. It follows from the construction that
(\ref{eq-key777}) will hold for $A_u$.
Indeed, if $\psi_{\CI}-\psi_{\CI}^*= c_1 (1,\ldots,1)^{\dagger}$,
then $\psi_{\CI}= c_2 \psi_{\CI}^*$, which in turn means that
$\psi_{\CE}= c_2 \psi_{\CE}^*$; therefore, the corresponding
service rates are $\sum_j \mu_{ij} \psi_{ij} = c_2 \sum_j \mu_{ij}
\psi_{ij}^*=
c_2 \lambda_i = c_2$ for all $i$; therefore, $\dot{\psi}_{\CI} =
(1-c_2) (1,\ldots,1)^{\dagger}$.
\end{pf}
%
%
\begin{theorem}
\label{th-888}
Suppose we have a system with parameters satisfying Lem\-ma~\ref{lem-777},
in the Halfin--Whitt regime,\vspace*{1pt} described in this section.
Then, the sequence of stationary distributions
of $\hat X_\CI^r$ (and of $\hat\Psi_\CI^r$) escapes to infinity:
the measure of any compact set vanishes.
\end{theorem}
\begin{pf}
Since $(1,\ldots,1)^{\dagger}$ is an eigenvector of $A_u$, for any
$y\in\BR^I$ we have
\[
\pi A_u F[y] = \pi A_u \pi y = A_c \pi y.
\]
Then, taking the $\pi$-projection of equation (\ref{eq-psi-limit-hw}),
we see that
$\pi\hat X_\CI$ satisfies the following \textit{linear} SDE
%
%
\begin{equation}
\label{eq-SDE-888} \pi\hat X_\CI(t) = \pi\hat X_\CI(0) +
\int_0^t A_c \pi\hat
X_\CI(s) \,ds + \pi\bigl(\sqrt{2\lambda_i}
B_i(t)\bigr).
\end{equation}
Given instability of linear equation
(\ref{eq-SDE-888}), we can repeat the argument
of Section~\ref{sectioninvariantmeasures} to show that the sequence of
projections of the
stationary distributions of $\hat X_\CI^r$ on $L$ escapes to infinity.
\end{pf}

\subsection{Tightness of stationary distributions
in the case when service rate depends on the server type only}
\label{sectiontightness}

In this section we consider a special case when there exists
a set of positive rates $\{\mu_j\}$, such that $\mu_{ij}=\mu_j$
as long as $(ij)\in\CE$.
We demonstrate tightness of invariant
distributions.
(An analogous result holds for the underload system, $\rho<1$, as
sketched out at the end of this section.)
This, in combination with the transient diffusion limit results,
allows us to claim that the limit of invariant distributions
is the invariant distribution of the limiting diffusion process.
%
%
\begin{theorem}
\label{th-tightness-spec-case}
Suppose $\mu_{ij}=\mu_j, (ij)\in\CE$
and $\rho=1$.
Consider a system under the LQFS-LB rule
in the asymptotic regime defined above in this section. Then, for any real
\[
\theta< \theta_0:= \frac{2 \min_i \lambda_i}{\sum_i \lambda_i +
(\max_j \mu_j)\sum_j \beta_j},
\]
the stationary distributions are such that
\[
\lim\sup_r \BE\biggl[\sum_i \exp
\bigl(\theta\hat Q^r_i\bigr) + \sum
_j \beta_j \exp\bigl(\theta\hat
Z^r_j/\beta_j\bigr)\biggr] < \infty.
\]
\end{theorem}
\begin{pf}
Note that the statement is trivial for $\theta= 0$. Also, for $\theta
> 0$ each term $\exp(\theta\hat Z^r_j/\beta_j)$ is bounded so has
finite expectation, while for $\theta< 0$ each term $\exp(\theta\hat
Q^r_i)$ is bounded so has finite expectation.

Our method is related to that in~\cite{GamarnikStolyar}. (The
exposition below is self-contained.)

\textit{Step} 1: \textit{Preliminary bounds.} Consider the embedded
Markov chain
taken at the instants of (say, right after) the transitions.
We will use uniformization, that is, we keep the total rate of all
transitions from any state constant at\vadjust{\goodbreak}
$\alpha^r r = \sum_i \lambda^r_i + \sum_j r \beta_j \mu^*$,
where $\mu^* = \max\mu_j$;
note that, as $r\to\infty$, $\alpha^r \to
\alpha^*=\sum_i \lambda_i + \sum_j \beta_j \mu^*$.
The transitions are of three types: arrivals, departures and virtual
transitions, which do not change the state of the system.
The rate of a transition due to a type $i$ arrival
is $\lambda_i^r$; for the service completion at pool $j$ the rate
is $\mu_j (r\beta_j + Z^r_j)$ (recall\vspace*{1pt} $Z^r_j \leq0$); and a virtual
transition
occurs at the complementary rate
$\alpha^r r - \sum_i \lambda^r_i - \sum_j \mu_j (r\beta_j + Z^r_j)$.
(Obviously, the probability that a transition occurring
at a transition instant has a given type is the ratio of the
corresponding rate and
$\alpha^r r$.)
The stationary distribution of the embedded Markov chain is the same as
that of the original, continuous-time chain.

In the rest of the proof,
$\tau\in\{0,1,2,\ldots\}$ refers to the discrete time of
the embedded Markov chain.

We will work with the following Lyapunov function:
%
%
\begin{equation}
\label{eqnLyapunovfunction} \CL(\tau):= \sum_i
\exp\bigl(\theta\hat Q^r_i(\tau)\bigr) + \sum
_j \beta_j \exp\bigl(\theta\hat
Z^r_j(\tau)/\beta_j\bigr).
\end{equation}

Throughout, we use the bound
%
%
\begin{equation}
\label{eqnTaylorexponential} \exp(\theta y) \leq\exp(\theta x)
\bigl(1 +
\theta(y-x) + \tfrac12 \theta^2 (y-x)^2 \exp\bigl(\theta
\llvert y-x\rrvert\bigr) \bigr),
\end{equation}
which arises from the second-order Taylor expansion of $\exp(\theta y)$.

A priori we do not know that $\BE[\CL(\tau)]$ exists for $\theta>
0$. Indeed, while $\hat Z^r_j(t)$ is bounded for any $r$ (above by 0
and below by $-\beta_j \sqrt{r}$), the scaled queue size $\hat
Q^r_i(t)$ is unbounded. To deal with this,
we also consider the truncated Lyapunov function $\CL^K = \min\{\CL
,K\}$.

In the equation below, let $x$ denote the variable
of interest (either $\hat Q^r_i$ or $\hat Z^r_j/\beta_j$), and let
$S(\tau)$ denote the state of the embedded Markov chain at time~$\tau$.
 From (\ref{eqnTaylorexponential}) we obtain
\begin{eqnarray*}
\hspace*{-3pt}&&
\BE\bigl[\exp\bigl(\theta x(\tau+1)\bigr) - \exp\bigl(\theta
x(\tau
)\bigr) \vert
S(\tau)\bigr]
\\
\hspace*{-3pt}&&\qquad\leq\exp\bigl(\theta x(\tau)\bigr) \bigl(\theta\BE\bigl
[x(\tau
+1)-x(\tau)
\vert S(\tau)
\bigr]
\\
\hspace*{-3pt}&&\hspace*{52.5pt}\qquad\quad{}+\tfrac12 \theta^2\BE\bigl[\bigl(x(\tau+1)-x(\tau)
\bigr)^2 \exp\bigl(\theta\bigl\llvert x(\tau+1)-x(\tau)\bigr
\rrvert
\bigr) \vert S(\tau) \bigr] \bigr).
\end{eqnarray*}
Since for both $\hat Z^r_j$ and $\hat Q^r_i$ the change
in a single transition is bounded by $1/\sqrt{r}$, we conclude
%
%
\begin{eqnarray}
\label{eqnsecond-orderdrift}
&&
\BE\bigl[\exp\bigl(\theta\hat
Q_i^r(\tau+1)\bigr) - \exp\bigl(\theta\hat
Q_i^r(\tau)\bigr) \vert S(\tau)\bigr]
\nonumber\\
&&\qquad
\leq\exp\bigl(\theta\hat Q_i^r(\tau)\bigr) \biggl(\theta\BE
\bigl[\hat Q_i^r(\tau+1)-\hat Q_i^r(
\tau) \vert S(\tau)\bigr] \\
&&\qquad\quad\hspace*{94.1pt}{}+ \biggl(\frac12 \theta^2 \exp(\theta/
\sqrt{r}) \biggr) \frac1 r \biggr),
\nonumber
\\
\label{eqnsecond-orderdrift-z}
&&\BE\bigl[\beta_j \exp\bigl(\theta
\hat Z_j^r(\tau+1)/\beta_j\bigr) -
\beta_j \exp\bigl(\theta\hat Z_j^r(\tau)/
\beta_j\bigr) \vert S(\tau)\bigr]
\nonumber\\
&&\qquad\leq\exp\bigl(\theta\hat Z_j^r(\tau)/\beta_j
\bigr) \biggl(\theta\BE\bigl[\hat Z_j^r(\tau+1)-\hat
Z_j^r(\tau) \vert S(\tau)\bigr]\\
&&\qquad\quad\hspace*{93pt}{} + \biggl(
\frac
{1}{\beta_j}\frac12 \theta^2 \exp(\theta/ \sqrt{r}) \biggr)
\frac1 r \biggr).
\nonumber
\end{eqnarray}
Clearly, as long as values of $\theta$ are bounded,
for any fixed $C_2>1$ and all sufficiently (depending on $C_2$)
large $r$, the second summands in (\ref{eqnsecond-orderdrift}) and
(\ref{eqnsecond-orderdrift-z})
are upper bounded by $C_2 \frac12\theta^2 \frac1r$
and $\frac{1}{\beta_*} C_2 \frac12\theta^2 \frac1r$,
respectively, where $\beta_* = \min_j \beta_j$.
Note that the second bound is independent of $j$.

Next, we will obtain an upper bound on the drift
\[
\BE\bigl[\CL(\tau+1) - \CL(\tau) \vert S(\tau)\bigr].
\]
To do that, we introduce an artificial scheduling/routing rule, which
acts only within one time step, and is such that the increment
$\CL(\tau+1) - \CL(\tau)$ under this rule is ``almost'' a (pathwise,
w.p.1) upper bound on this increment under the actual---LQFS-LB---rule.
[It is important to keep in mind that the artificial rule is \textit{not}
a rule that is applied continuously. It is limited to one time step,
and its sole purpose is to derive a pathwise upper bound on the
increment $\CL(\tau+1) - \CL(\tau)$ within one time step.]

\textit{Step} 2: \textit{Artificial scheduling/routing rule.}
We will use the following notation:
$\CI_+=\CI_+(\tau):= \{i\dvtx\hat Q^r_i(\tau) > 0\}$,
$\CI_0=\CI_0(\tau):= \{i\dvtx\hat Q^r_i(\tau) = 0\}$,
$\CJ_-=\CJ_-(\tau):= \{j\dvtx\hat Z^r_j(\tau) < 0\}$,
$\CJ_0=\CJ_0(\tau):= \{j\dvtx\hat Z^r_j(\tau) = 0\}$.\vspace*{2pt}

\textit{Scheduling}: Departures from servers $j \in\CJ_-$
are processed normally, that is, reduce the corresponding $Z^r_j(\tau
)$ by 1.
Whenever there is a departure from a server pool $j \in\CJ_0$,
the server takes up a customer of type $i$ with probability $\lambda
^r_{ij} / \sum_i \lambda^r_{ij}$, keeping $Z^r_j(\tau+1) = 0$ and
reducing $Q^r_i(\tau+1) = Q^r_i(\tau) - 1$.
However, if it happens that the chosen $i$ is such that $Q^r_i(\tau)=0$,
that is, $i\in\CI_0$, then we keep $Q^r_i(\tau+1) = Q^r_i(\tau) =
0$ and instead allow $Z^r_j(\tau+1) = -1$.\vspace*{2pt}

\textit{Routing}: Arrivals to customer types $i \in\CI_+$
are processed normally, that is, the corresponding $Q^r_i(\tau)$ is
increased by 1.
Whenever there is an arrival to a customer type $i \in\CI_0$,
it is routed to server pool $j$ with probability
$\lambda^r_{ij} / \lambda^r_i$, keeping $Q^r_i(\tau+1) = Q^r_i(\tau
) = 0$ and
increasing $Z^r_j(\tau+1) = Z^r_j(\tau) + 1$.
However, if it happens that the chosen $j$ is such that $Z^r_j(\tau)=0$,
that is, $j\in\CJ_0$, then we keep $Z^r_j(\tau+1) = Z^r_j(\tau) =
0$ and instead allow $Q^r_i(\tau+1) = 1$.

\textit{Step} 3: \textit{One time-step drift under the artificial rule.}
For $i \in\CI_+$,
\[
\BE\bigl[\hat Q^r_i(\tau+1) - \hat Q^r_i(
\tau) \vert S(\tau)\bigr] = \frac{1}{\alpha^r r} \frac{1}{\sqrt{r}}
\biggl(
\lambda_i^r - \sum_j (
\mu_j r \beta_j) \frac{\lambda^r_{ij}}{\sum_k \lambda^r_{kj}}
\biggr)
\]
or, recalling that
%
%
\begin{equation}
\label{eq-nominal-point-diffusion} \sum_k
\lambda_{kj}^r = \mu_j \beta_j r
\rho^r =\mu_j \beta_j r (1-C/\sqrt{r}),
\end{equation}
we obtain
%
%
\begin{equation}
\label{eqndriftofQinI^*} \BE\bigl[\hat Q^r_i(
\tau+1) - \hat Q^r_i(\tau) \vert S(\tau)\bigr] = -
\frac{C \lambda_i}{\alpha^*} \frac{1+o(1)}{r},\qquad i \in\CI_+,
\end{equation}
where $o(1)$ is a fixed function, vanishing as $r\to\infty$.

If $\hat Q^r_i(\tau)=0$ (i.e., $i \in\CI_0$),
and a new type $i$ arrival is routed to pool $j$ with $\hat Z^r_j(\tau)<0$
(i.e., $j \in\CJ_-$), then of course $\hat Q^r_i$ stays at $0$
and $\hat Q^r_i(\tau+1) - \hat Q^r_i(\tau)=0$.
However,\vspace*{1pt} if a new type $i$ arrival has to be routed to $j\in\CJ_0$,
then (by the definition of artificial rule)
$\hat Q^r_i(\tau+1) - \hat Q^r_i(\tau)=\hat Q^r_i(\tau+1)=1/\sqrt{r}$.
Thus, we can write
%
%
\begin{equation}
\label{eqndriftofQinI0} \BE\bigl[\hat Q^r_i(
\tau+1) - \hat Q^r_i(\tau) \vert S(\tau)\bigr] = \sum
_{j\in\CJ_0} \frac{\lambda^r_{ij}}{\alpha^r r} \frac{1}{\sqrt
{r}},\qquad i \in
\CI_0.
\end{equation}
Note that the right-hand side of (\ref{eqndriftofQinI0}) is of order
$1/\sqrt{r}$, not $1/r$. However, we will see shortly
that order $1/\sqrt{r}$ terms in
$\BE[\CL(\tau+1) - \CL(\tau) \vert S(\tau)]$ cancel out,
and this expected drift is in fact of order $1/r$.

The treatment of the drift of $\hat Z^r_j$ is similar
[and again makes use of (\ref{eq-nominal-point-diffusion})].
We obtain
%
%
\begin{eqnarray}\qquad\quad
\label{eqndriftofZinJ*}
\BE\bigl[\hat Z^r_j(
\tau+1) - \hat Z^r_j(\tau) \vert S(\tau)\bigr] &=& -
\frac{1}{\alpha^r } \mu_j \bigl(\hat Z^r_j(
\tau)+\beta_j C\bigr) \frac{1}{r},\qquad j \in\CJ_-,
\\
\label{eqndriftofZinJ0new}
\BE\bigl[\hat
Z^r_j(\tau+1) - \hat Z^r_j(
\tau) \vert S(\tau)\bigr]
&=&-\frac{1}{\sqrt{r}} \sum_{i\in\CI_0} \frac{r\mu_j
\beta
_j}{\alpha^r r}
\frac{\lambda^r_{ij}}{\sum_k \lambda^r_{kj}} \nonumber\\[-8pt]\\[-8pt]
&=& - \frac
{1}{1-C/\sqrt{r}}
\sum_{i\in\CI_0}
\frac{\lambda^r_{ij}}{\alpha^r r} \frac{1}{\sqrt{r}},\qquad j \in\CJ_0.
\nonumber
\end{eqnarray}
We can rewrite (\ref{eqndriftofZinJ0new}) as
%
%
\begin{eqnarray}
\label{eqndriftofZinJ0}
&&\BE\bigl[\hat Z^r_j(
\tau+1) - \hat Z^r_j(\tau) \vert S(\tau)\bigr]\nonumber\\[-8pt]\\[-8pt]
&&\qquad = -\sum
_{i\in\CI_0} \frac{\lambda^r_{ij}}{\alpha^r r} \frac{1}{\sqrt
{r}} -
\frac{C \sum_{i\in\CI_0} \lambda_{ij}}{\alpha^*} \frac
{1+o(1)}{r},\qquad j \in
\CJ_0,\nonumber
\end{eqnarray}
where $o(1)$ is a fixed function, vanishing as $r\to\infty$.

Note that if $\CL(\tau) \geq K$, then $\CL^K(\tau+1) - \CL^K(\tau
) \le0$,
and if $\CL(\tau) < K$, then $\CL^K(\tau+1) - \CL^K(\tau) \le\CL
(\tau+1) - \CL(\tau)$.
Putting together this observation and
equations
(\ref{eqnsecond-orderdrift}), (\ref{eqnsecond-orderdrift-z}),
(\ref{eqndriftofQinI^*})--(\ref{eqndriftofZinJ0}),
we obtain
%
%
\begin{subequation}
\label{eq-key-drift-estimate}
%
%
\begin{eqnarray}\qquad\quad
&&
\BE\bigl[\CL^K(\tau+1) - \CL^K(\tau) \vert S(\tau)
\bigr]
\\
\label{eqnfirst-orderdriftQ}
&&\qquad\leq\one_{\{\CL(\tau) \leq K\}} \biggl( \sum_{i \in\CI_+}
\exp
\bigl(\theta\hat Q^r_i(\tau)\bigr)\theta\biggl[-
\frac
{C\lambda_i (1+o(1))}{\alpha^*}\biggr] \frac{1}{r}
\\
\label{eqnI0,J0}
&&\hspace*{46.5pt}\qquad\quad{}+ \sum_{i \in\CI_0, j \in\CJ_0} \theta\lambda^r_{ij}
\frac
{1}{\alpha^r r}\frac{1}{\sqrt{r}}
\\
\label{eqnfirst-orderdriftZ}
&&\hspace*{46.5pt}\qquad\quad{}+ \sum_{j \in\CJ_-} \exp\bigl(\theta\hat
Z^r_j(\tau)/\beta_j\bigr) \theta\biggl[-
\frac{\mu_j}{\alpha^r}\biggr] \bigl[\hat Z^r_j(\tau)+
\beta_j C\bigr] \frac
{1}{r}
\\
\label{eqnJ0,I0}
&&\hspace*{46.5pt}\qquad\quad{}+ \sum_{j \in\CJ_0, i \in\CI_0} \theta\biggl
[-\lambda^r_{ij}
\frac
{1}{\alpha^r r}\frac{1}{\sqrt{r}} -\frac{C\lambda_i
(1+o(1))}{\alpha
^*}\frac{1}{r}\biggr]
\\
\label{eqn2ndorderQerror}
&&\hspace*{46.5pt}\qquad\quad{}+ \sum_{i \in\CI} \exp\bigl(\theta\hat
Q^r_i(\tau)\bigr) \biggl(\frac
{C_2}{2}
\theta^2 \biggr) \frac1 r
\\
\label{eqn2ndorderZerror}
&&\hspace*{106.5pt}\qquad\quad{}+\sum_{j \in\CJ} \frac{1}{\beta_*}\exp\bigl
(\theta
\hat Z^r_j(\tau)/\beta_j\bigr) \biggl(
\frac{C_2}{2} \theta^2 \biggr) \frac1 r \biggr).
\end{eqnarray}
\end{subequation}
%
Note that the $O(1/\sqrt{r})$ terms in (\ref{eqnI0,J0}) and
(\ref{eqnJ0,I0}) cancel each other as promised, so there are no
$O(1/\sqrt{r})$ terms in the final bound.

\textit{Step} 4: \textit{One time-step drift under the LQFS-LB rule.}
We now explain in what sense the increment $\CL(\tau+1) - \CL(\tau
)$ under the
artificial rule
is ``almost'' an upper bound on this increment under LQFS-LB.
To illustrate the idea, suppose first that all $\beta_j$ are equal.
Then, it is easy to observe that
for any fixed $S(\tau)$, the increment $\CL(\tau+1) - \CL(\tau)$ under
the artificial rule is (with probability 1)
an upper bound of this increment under LQFS-LB.
Indeed, suppose first that a transition of the Markov chain is associated
with a service completion
in server pool $j$ with $\hat Z_j^r=0$. (If $\hat Z_j^r<0$, there is no
difference
in what the two rules do.)\vspace*{1pt} The only case of interest is when
the LQFS-LB ``takes'' a new customer for service from queue $i$ with
\mbox{$\hat Q_i^r>0$},
while the artificial rule tries to take a customer from a different
queue~$i'$.
Then $\hat Q_i^r \ge\hat Q_{i'}^r$ must hold, with $\hat Q_i^r > \hat Q_{i'}^r$
being the nontrivial case.
If $\hat Q_{i'}^r>0$, then the
LQFS-LB
will decrease the larger
queue, and so the increment $\CL(\tau+1) - \CL(\tau)$
under the LQFS-LB is smaller (which is true for both positive and
negative $\theta$).
If $\hat Q_{i'}^r=0$, then the
LQFS-LB
will still decrease queue $\hat Q_i^r$, while the artificial rule will instead
decrease $\hat Z_j^r$; using convexity of $e^{\theta x}$, we verify
that, again,
the increment $\CL(\tau+1) - \CL(\tau)$
under the LQFS-LB is smaller (for both positive and negative $\theta$).
If transition of the Markov chain is associated
with a new customer arrival, we use an analogous argument to show that,
again, the increment $\CL(\tau+1) - \CL(\tau)$
under the LQFS-LB cannot be greater than that under the artificial rule.
We conclude that when all $\beta_j$ are equal,
the key estimate (\ref{eq-key-drift-estimate}) of the
espected drift holds, in exactly same form, for LQFS-LB rule as well.

Now consider the case of general $\beta_j$. In the event of a service
completion
(and then possibly taking a customer for service from one of the
nonzero queues),
the increment $\CL(\tau+1) - \CL(\tau)$ under LQFS-LB is still no greater
than under the artificial rule. (Verified similarly to the case of all
$\beta_j$ being equal.)
The only situation when LQFS-LB can possibly cause
a greater increment than the artificial rule is as follows. There is an
arrival of
a type $i$ customer, which the artificial rule routes to pool $j$ with
$\hat Z_j^r <0$,
but the LQFS-LB will instead route it to pool $k$ such that $\hat
Z_j^r/\beta_j \ge
\hat Z_k^r/\beta_k$. Given convexity of function $e^{\theta x}$, the
``worst case,'' that is,
the largest increment of $\CL(\tau+1) - \CL(\tau)$, occurs when
$\hat Z_k^r$ is such that
the equality
holds, $\hat Z_j^r/\beta_j = \hat Z_k^r/\beta_k$.
(If $\theta>0$ the positive increment
gets larger, if we were to increase $\hat Z_k^r$; if $\theta<0$ the negative
increment gets smaller in absolute value, if we were to increase $\hat Z_k^r$.
Note also that here we allow $\hat Z_k^r$, determined by the equality,
to be such that $Z_k^r= \hat Z_k^r \sqrt{r}$ is possibly noninteger,
because we only use this value
of $\hat Z_k^r$ to
estimate the increment of a function.)
Thus, as we replace the artificial rule by LQFS-LB,
in the ``worst case,'' the increment
\[
\beta_j \exp\bigl(\theta\bigl[\hat Z^r_j(
\tau)+r^{-1/2}\bigr]/\beta_j\bigr) - \beta_j
\exp\bigl(\theta\hat Z^r_j(\tau)/\beta_j
\bigr)
\]
may need to be replaced by
\[
\beta_k \exp\bigl(\theta\bigl[\hat Z^r_k(
\tau)+r^{-1/2}\bigr]/\beta_k\bigr) - \beta_k
\exp\bigl(\theta\hat Z^r_k(\tau)/\beta_k
\bigr)
\]
with $\hat Z_k^r(\tau)$ satisfying $\hat Z_j^r(\tau)/\beta_j = \hat
Z_k^r(\tau)/\beta_k$.
In this case
we obtain
%
%
\begin{eqnarray}
\label{eqnsecond-orderdrift-z-special}
&&
\beta_k \exp\bigl(\theta
\hat Z_k^r(\tau+1)/\beta_k\bigr) -
\beta_k \exp\bigl(\theta\hat Z_k^r(\tau)/
\beta_k\bigr)
\nonumber
\\
&&\qquad\leq\exp\bigl(\theta\hat Z_k^r(\tau)/\beta_k
\bigr) \biggl(\theta r^{-1/2} + \biggl(\frac{1}{\beta_k}\frac12
\theta^2 \exp(\theta/ \sqrt{r}) \biggr) \frac1 r \biggr)
\\
&&\qquad\leq\exp\bigl(\theta\hat Z_j^r(\tau)/\beta_j
\bigr) \biggl(\theta r^{-1/2} + \biggl(\frac{1}{\beta_*}\frac12
\theta^2 \exp(\theta/ \sqrt{r}) \biggr) \frac1 r \biggr),
\nonumber
\end{eqnarray}
%
This means that, \textit{under LQFS-LB rule},
\textit{the estimate} (\ref{eq-key-drift-estimate}) \textit{still holds.}

\textit{Step} 5: \textit{Exponential moments estimates.}
Next, note that for each fixed $K>0$ and each fixed parameter $r$,
the values of $\exp(\theta\hat Q^r_i(\tau))$ are uniformly bounded
over all
states $S(\tau)$ satisfying condition $\CL(\tau) \leq K$; the values of
$\exp(\theta\hat Z^r_j(\tau)/\beta_j)$ are ``automatically''
uniformly bounded (for a fixed $r$).
We take the expected values of both parts of (\ref{eq-key-drift-estimate})
with respect to the invariant distribution. The expectation of the
left-hand side is of course $0$,
and so we get rid of the factor $1/r$ from the right-hand side
expectation. The resulting
estimates we will write separately for the cases $\theta>0$ and
$\theta<0$
(with the case $\theta=0$ being trivial).

\textit{Case $\theta> 0$.}
For a fixed $\theta>0$, the expected value of the sum of all terms not
containing
$\exp(\theta\hat Q^r_i(\tau))$ is bounded (uniformly in $r$).
Indeed, this follows from the
facts that $\hat Z^r_j(\tau)\le0$ and
$0\le-\theta\hat Z^r_j(\tau)\exp(\theta\hat Z^r_j(\tau)/\beta_j)
\le\beta_j/e$
(because $0 \geq x e^x \geq-\frac1e$ for $x \le0$). Then, we obtain
%
%
\begin{equation}
\label{eq-q-bound1} \BE\biggl[\one_{\{\CL(\tau) \leq K\}}\sum
_{i \in\CI_+}\exp\bigl(\theta\hat Q^r_i(
\tau)\bigr) \biggl(\frac{C\lambda_i (1+o(1))}{\alpha^*}\theta-
\biggl
(\frac
{C_2}{2}
\theta^2 \biggr) \biggr) \biggr] \leq C_1\hspace*{-35pt}
\end{equation}
for some constant $C_1 = C_1(\theta) > 0$, uniformly on all
sufficiently large $r$.
Now let us fix a sufficiently small positive $\theta$, so that
all coefficients of $\exp(\theta\hat Q^r_i(\tau))$ are at least some
$\varepsilon>0$
(for all large $r$).
Recalling that $C_2>1$ can be arbitrarily close to~$1$, it suffices that
$\theta< \theta_0 = 2 (\min_i \lambda_i) /\alpha^*$.
Then,
\[
\BE\biggl[\one_{\{\CL(\tau) \leq K\}}\sum_{i \in\CI_+}\exp
\bigl(\theta\hat Q^r_i(\tau)\bigr) \biggr] \leq
C_1/\varepsilon,
\]
from where, letting $K\to\infty$, by monotone convergence, we obtain
%
%
\begin{equation}
\label{eqnexpectationofQ} \BE\biggl[\sum_{i \in\CI_+}
\exp\bigl(\theta\hat Q^r_i(\tau)\bigr) \biggr] \leq
C_1/\varepsilon< \infty,
\end{equation}
uniformly on all large $r$.

\textit{Case $\theta< 0$.} Fix arbitrary $\theta<0$. In this case,
the expected value of the sum of all terms not containing
$\exp(\theta\hat Z^r_j(\tau))$, is bounded (uniformly on $r$). We
can write
%
%
\begin{eqnarray}
\label{eq-z-bound1}
&&
\BE\biggl[\one_{\{\CL(\tau) \leq K\}}\sum
_{j \in\CJ_-}\exp\bigl(\theta\hat Z^r_j(
\tau)/\beta_j\bigr)
\nonumber\\[-8pt]\\[-8pt]
&&\hspace*{77pt}{}\times\biggl(\theta\biggl[\frac{\mu_j}{\alpha
^r} \biggr
] \bigl[\hat Z^r_j(
\tau)+\beta_j C \bigr] - \biggl(\frac{1}{\beta_*} \frac{C_2}{2}
\theta^2 \biggr) \biggr)\biggr] \leq C'_1
\nonumber
\end{eqnarray}
for some constant $C'_1 = C'_1(\theta) > 0$, uniformly on all
sufficiently large $r$.
Let us choose sufficiently large $K_1>0$, such that the condition
$\hat Z^r_j(\tau)\le- K_1$ implies that
\[
\biggl(\theta\biggl[\frac{\mu_j}{\alpha^r}\biggr] \bigl[\hat Z^r_j(
\tau)+\beta_j C\bigr] - \biggl(\frac{1}{\beta_*} \frac{C_2}{2}
\theta^2 \biggr) \biggr) \ge\varepsilon
\]
for some $\varepsilon>0$ (and all large $r$). Then, from (\ref{eq-z-bound1}),
\[
\BE\biggl[\one_{\{\CL(\tau) \leq K\}}\sum_{j \in\CJ_-}
\one_{\{
\hat Z^r_j(\tau)\leq-K_1\}} \exp\bigl(\theta\hat Z^r_j(\tau)/
\beta_j\bigr) \biggr] \leq C'_1/\varepsilon
\]
from where, letting $K\to\infty$, by monotone convergence, we obtain
\[
\BE\biggl[\sum_{j \in\CJ_-}
\one_{\{ \hat Z^r_j(\tau)\leq-K_1\}} \exp\bigl(\theta\hat
Z^r_j(\tau)/
\beta_j\bigr) \biggr] \leq C'_1/\varepsilon<
\infty,
\]
uniformly on all large $r$, which implies the required result.
\end{pf}
%
%
\begin{corollary}
The sequence of stationary distributions of the processes $ ((\hat
Q^r_i(\cdot)),(\hat Z^r_j(\cdot)) )$ has a weak limit, which is
the unique stationary distribution of the
limiting process $ ((\hat Q_i(\cdot)),(\hat Z_j(\cdot)) )$,
described as follows:
\[
\hat Q_i(t) = \max\bigl\{\hat Y(t)/I,0\bigr\}\qquad\forall i,\qquad
\hat
Z_j(t) = \min\biggl\{\frac{\beta_j}{\sum_k\beta_k} \hat
Y(t),0\biggr\}
\qquad\forall
j,
\]
where $\hat Y(\cdot)$ is a one-dimensional diffusion process
with constant variance parameter $2\sum_i \lambda_i$ and piece-wise
linear drift,
equal at point $x$ to
\[
-\biggl[\sum_j \mu_j\biggr]
\bigl[C+\min\{x,0\}\bigr].
\]
The invariant distribution density is then a continuous function,
which is a ``concatenation'' at point $0$ of
exponential (for $x\ge0$)
and Gaussian (for $x\le0$) distribution densities.
\end{corollary}
\begin{pf}
Theorem~\ref{th-tightness-spec-case} of course implies
tightness of stationary distributions of $ ((\hat Q^r_i(\cdot
)),(\hat Z^r_j(\cdot)) )$.
Then it follows from~\cite{LiptserShiryaev}, Theorem 8.5.1 (whose
conditions are easily verified
in our case), that as $r\to\infty$,
any weak limit of the sequence of stationary distributions of
the processes $ ((\hat Q^r_i(\cdot)),(\hat Z^r_j(\cdot)) )$
is a stationary distribution of the limit process,
described in~\cite{GurvichWhitt}, Theorem 4.4,
and therefore is the one-dimensional diffusion specified in the
statement of the corollary.
\end{pf}

Finally, we remark that a tightness result analogous to
Theorem~\ref{th-tightness-spec-case} holds for the underloaded system,
$\rho<1$,
and can be proved essentially the same way.

The asymptotic regime in this case is such that $\lambda^r_i = r
\lambda_i$
[there is no point in considering $O(\sqrt{r})$ terms in $\lambda^r_i$
when $\rho<1$]. We denote $Z^r_j(t) = \Psi^r_j(t) - r \beta_j \rho$
(which is consistent with the definition given earlier in this section
for $\rho=1$), and keep notation $Q^r_i(t)$ for the queue length.
We work with the following Lyapunov function:
\[
\CL:= \sum_i \bigl[\exp
\bigl(\theta(1-\rho) \sqrt{r} +\theta\hat Q^r_i\bigr) -
\exp\bigl(\theta(1-\rho) \sqrt{r}\bigr) \bigr] + \sum
_j \beta_j \exp\bigl(\theta\hat
Z^r_j/\beta_j\bigr). 
\]
The same approach
as in the proof of Theorem~\ref{th-tightness-spec-case}
leads to the following result: for any real $\theta$,
\[
\lim\sup_r \BE\biggl[\sum_j
\exp\bigl(\theta\hat Z^r_j\bigr)\biggr] < \infty.
\]
The limiting\vspace*{1pt} process for $(\hat Z^r_j(\cdot))$ is
$(\hat Z_j(\cdot)) = (\frac{\beta_j}{\sum_k\beta_k} \hat Y(\cdot))$,
with $\hat Y(\cdot)$ being a one-dimensional
Ornstein--Uhlenbeck process, with Gaussian stationary distribution.
The limit of stationary distributions of $(\hat Z^r_j(\cdot))$
is the stationary distribution of $(\hat Z_j(\cdot))$.

\section*{Acknowledgments}

The authors would like to thank the referees for useful comments that
helped to improve the exposition of the material.


%

\printaddresses

\end{document}